\documentclass[12pt, reqno]{amsart}

\usepackage{amssymb,amsmath,amsthm,graphicx,enumitem,systeme, array,enumitem,commath,mathtools,multicol,url}
\usepackage{setspace,blindtext}
\usepackage{varioref}
\usepackage[colorlinks=true,linkcolor=blue]{hyperref}
\usepackage[capitalise]{cleveref}
\usepackage{blindtext}
\usepackage{fancyhdr}
\usepackage{pgfplots}
\usepackage{geometry}
\usepackage{xcolor}
\usepackage{mathabx}

\usepackage{caption}
\allowdisplaybreaks 

\usepackage{aliascnt}
\usepackage{varioref}
\usepackage[colorlinks=true,linkcolor=blue]{hyperref}
\usepackage[capitalise]{cleveref} 

\theoremstyle{plain}

\newtheorem{theorem}{Theorem}[section]

\newaliascnt{lemma}{theorem}
\newtheorem{lemma}[lemma]{Lemma}
\aliascntresetthe{lemma}

\newaliascnt{corollary}{theorem}
\newtheorem{corollary}[corollary]{Corollary}
\aliascntresetthe{corollary}

\newaliascnt{proposition}{theorem}

\aliascntresetthe{proposition}

\theoremstyle{definition}

\newaliascnt{definition}{theorem}

\aliascntresetthe{definition}

\newaliascnt{example}{theorem}

\aliascntresetthe{example}

\newaliascnt{remark}{theorem}
\newtheorem{remark}[remark]{Remark}
\aliascntresetthe{remark}

\newaliascnt{conjecture}{theorem}

\aliascntresetthe{conjecture}

\title[Area between sines and cosines]{Asymptotic areas between powers of \\sine and cosine curves}

\author{Atiratch Laoharenoo}
\address{Department of Mathematics and Computer Science, Kamnoetvidya Science Academy, Rayong 21210, Thailand}
\email{atiratch.l@kvis.ac.th}

\author{Chanatip Sujsuntinukul}
\address{Department of Mathematics, The University of Hong Kong, Pokfulam, Hong Kong}
\email{chanatip@connect.hku.hk}



\begin{document}

\maketitle

\begin{abstract}
Motivated by Dombrowski and Dresden's work in 2025, we find the exact values of the limits
\[\lim_{k\to\infty}\int_0^{\rho\pi}|\sin^n(kx)-\sin^nx|dx\quad\text{and}\quad \lim_{k\to\infty}\int_0^{\rho\pi}|\cos^n(kx)-\cos^nx|dx\]
for $k,n\in\mathbb{N}$ and $\rho\ge 0$. In addition, we provide several simple recursive formulas which relate these integrals together. The key technique is to locate the zeros of the integrands explicitly, which allows removal of the absolute value and reduces the problem to evaluating limits of telescoping trigonometric sums via asymptotic analysis. 
\end{abstract}

\section{Introduction and main results}

The study of integrals involving trigonometric functions has long held significance in mathematical analysis, with numerous techniques developed for their evaluation. For instance, see \cite{DB}, \cite{LK}, \cite{AL}, \cite{Sofo}. Our main interest lies in the recent work by Dombrowski and Dresden in
 2025 \cite{MG}. They explicitly found the area between the functions $\cos^nx$ and $\cos^n(kx)$ over $[0, \pi]$ as $k\to\infty$, where $k,n\in\mathbb{N}$. To be more precise, they evaluated the definite integral
\begin{align*}\label{one}
    \mathcal{A}_n:=\lim_{k\to\infty}\int_0^{\pi}|\cos^nx-\cos^n(kx)|dx.
\end{align*}
When $n$ is odd, they found that
\[\mathcal{A}_n=\frac{1}{2^{n-4}\pi}\sum_{m=0}^{(n-1)/2}\binom{n}{m}\frac{1}{(n-2m)^2}.\]
On the other hand, when \(n\) is even, the value of \(\mathcal{A}_n\) depends on the residue of \(n\) modulo \(4\), and there are two possible cases as follows.
\begin{align*}
    \mathcal{A}_n=\begin{cases}
    \displaystyle\frac{1}{2^{n-4}\pi}\sum_{m=0}^{(n-2)/4}\binom{n}{2m}\frac{1}{(n/2-2m)^2} &\text{if } n\equiv 2\pmod{4},\\\\
    \displaystyle\frac{1}{2^{n-4}\pi}\sum_{m=0}^{(n-4)/4}\binom{n}{2m+1}\frac{1}{(n/2-(2m+1))^2} &\text{if } n\equiv 0 \pmod{4}.
    \end{cases}
\end{align*}
Below, we state some values of $\mathcal{A}_n$ when $n$ are small.
\begin{align*}
    \mathcal{A}_1=\frac{8\cdot 1}{(1)^2\pi}, \quad \mathcal{A}_2=\frac{16\cdot 1}{(2)^2\pi}, \quad \mathcal{A}_3=\frac{8\cdot 7}{(1\cdot 3)^2\pi}, \quad \mathcal{A}_4=\frac{16\cdot 16}{(2\cdot 4)^2\pi}.
\end{align*}
These values surprisingly turn out to have pleasant analytical and algebraic properties. For instance, they are somewhat closely related to the exponential generating functions of $(\arcsin x)/(1-x)$ and $(\arcsin x)^2/(2(1-x))$. 

In addition to these, they established a simple recursive formula for $\mathcal{A}_n$ as shown below.
\begin{align*}
    \mathcal{A}_n=\left(\frac{n-1}{n}\right)\mathcal{A}_{n-2}+\begin{cases}
        \displaystyle\frac{8}{n^2\pi} &\text{if $n$ is odd},\\
        \\
        \displaystyle\frac{16}{n^2\pi} &\text{if $n$ is even}.
    \end{cases}
\end{align*}

Motivated by their work, it is natural to evaluate \(\mathcal{A}_n\) in more general settings. For instance, instead of integrating over \([0,\pi]\), can we integrate the function over \([0,\rho\pi]\) where \(\rho \ge 0\)? Furthermore, can we replace the cosine functions appearing in the integrand of \(\mathcal{A}_n\) by sine functions? To this end, we introduce the following notations.
For $n, k\in\mathbb{N}$ and $\rho\in\mathbb{R}_{\ge 0}$, we set
\begin{align*}
    s_{n, k}(x):=\sin^n(kx)-\sin^nx\quad \text{and}\quad
    c_{n, k}(x):=\cos^n(kx)-\cos^nx.
\end{align*}
Consequently, we define four sequences $\{I_{n, k, \rho}\}$, $\{R_{n, k, \rho}\}$, $\{J_{n, k, \rho}\}$, $\{T_{n, k, \rho}\}$ over $k$ by
\begin{align*}
    I_{n, k, \rho}:=\int_0^{\rho\pi}|s_{n, k}(x)|dx,\quad
    R_{n, k, \rho}:=\int_{\frac{\pi}{2}}^{\rho\pi}|s_{n, k}(x)|dx,\\
     J_{n, k, \rho}:=\int_0^{\rho\pi}|c_{n, k}(x)|dx,\quad
    T_{n, k, \rho}:=\int_{\frac{\pi}{2}}^{\rho\pi}|c_{n, k}(x)|dx.
\end{align*}
(For $R_{n, k, \rho}$ and $T_{n, k, \rho}$, we further assume $\rho\ge 1/2$.)
Finally, we let
\begin{align*}
    I_{n, \rho}:=\lim_{k\to\infty}I_{n, k, \rho},\quad R_{n, \rho}:=\lim_{k\to\infty}R_{n, k, \rho},\quad
    J_{n, \rho}:=\lim_{k\to\infty}J_{n, k, \rho},\quad T_{n, \rho}:=\lim_{k\to\infty}T_{n, k, \rho}.
\end{align*}
In this work, we successfully compute the four limits above. Moreover, we are able to provide some recursive formulas relating to these limits. Below are our main results. Their proofs will be presented in later sections.

\subsection{Explicit values of $I_{n, \rho}$, $J_{n, \rho}$ for $\rho\in[0, 1/2]$}

\begin{theorem}\label{thm1.1}
    Let $\rho\in[0, 1/2]$ and $n\in\mathbb{N}$. 
    \begin{enumerate}
        \item If $n$ is odd, then
        \begin{align*}
            I_{n, \rho} &=\displaystyle\frac{(-1)^{\frac{n-1}{2}}}{2^{n-3}\pi}\sum_{m=0}^{(n-1)/2}(-1)^m\binom{n}{m}\frac{\sin((n-2m)\rho\pi)}{(n-2m)^2}\\
            \displaystyle&\qquad +(-1)^{\frac{n+1}{2}}\frac{\rho}{2^{n-2}}\sum_{m=0}^{(n-1)/2}(-1)^m\binom{n}{m}\frac{\cos((n-2m)\rho\pi)}{n-2m},\\
            J_{n, \rho} &=\frac{1}{2^{n-3}\pi}\sum_{m=0}^{(n-1)/2}\binom{n}{m}\frac{(1-\cos((n-2m)\rho\pi))}{(n-2m)^2}\\
            &\qquad-\frac{2\rho-1}{2^{n-1}}\sum_{m=0}^{(n-1)/2}\binom{n}{m}\frac{\sin((n-2m)\rho\pi)}{n-2m}.
        \end{align*}
        \item If $n$ is even, then
        \begin{align*}
            I_{n, \rho} &=\displaystyle\frac{(-1)^{\frac{n-2}{2}}}{2^{n-4}\pi}\sum_{m=0}^{(n-2)/2}(-1)^m\binom{n}{m}\frac{(1-\cos((n-2m)\rho\pi))}{(n-2m)^2}\\
            \displaystyle &\qquad +(-1)^{\frac{n}{2}}\frac{4\rho-1}{2^{n-1}}\sum_{m=0}^{(n-2)/2}(-1)^m\binom{n}{m}\frac{\sin((n-2m)\rho\pi)}{n-2m},\\
            J_{n, \rho}&=\frac{1}{2^{n-4}\pi}\sum_{m=0}^{(n-2)/2}\binom{n}{m}\frac{(1-\cos((n-2m)\rho\pi))}{(n-2m)^2}\\
            &\qquad-\frac{4\rho-1}{2^{n-1}}\sum_{m=0}^{(n-2)/2}\binom{n}{m}\frac{\sin((n-2m)\rho\pi)}{n-2m}.
        \end{align*}
    \end{enumerate}
\end{theorem}


These theorems are interesting because of their elegance and simplicity, and their proofs require several tools from calculus and analysis.

\subsection{Recursive formulas of  $I_{n, \rho}$, $J_{n, \rho}$ for $\rho\in[0, 1/2]$}

Alternatively, we can compute the values of $I_{n, \rho}$ and $J_{n, \rho}$ for $\rho\in[0, 1/2]$ using the following simple recursive formulas.

\begin{theorem}\label{thm1.2}
    Let $n\in\mathbb{N}_{\ge 3}$ and $\rho\in[0, 1/2]$. 
    \begin{enumerate}
        \item If $n$ is odd, then
        \begin{align*}
             I_{n, \rho}&=\left(\frac{n-1}{n}\right)I_{n-2, \rho}+\frac{4}{\pi n^2}\sin^n(\rho\pi)-\frac{2\rho}{n}\sin^{n-1}(\rho\pi)\cos(\rho\pi),\\
              J_{n, \rho}&=\left(\frac{n-1}{n}\right)J_{n-2, \rho}+\frac{4}{\pi n^2}(1-\cos^n(\rho\pi))-\frac{(2\rho-1)}{n}\cos^{n-1}(\rho\pi)\sin(\rho\pi).
        \end{align*}
        \item If $n$ is even, then
        \begin{align*}
            I_{n, \rho}&=\left(\frac{n-1}{n}\right)I_{n-2, \rho}+\displaystyle\frac{8}{\pi n^2}\sin^n(\rho\pi)-\left(\frac{4\rho-1}{n}\right)\sin^{n-1}(\rho\pi)\cos(\rho\pi),\\ 
            J_{n, \rho}&=\left(\frac{n-1}{n}\right)J_{n-2, \rho}+\frac{8}{\pi n^2}(1-\cos^n(\rho\pi))-\left(\frac{4\rho-1}{n}\right)\cos^{n-1}(\rho\pi)\sin(\rho\pi).
        \end{align*}
    \end{enumerate}
\end{theorem}


\subsection{Relations between $R_{n, \rho}$, $I_{n, \rho}$ and $J_{n, \rho}$, $T_{n, \rho}$ with $\rho\in[1/2, 1]$}

Here, we can compute the values of $I_{n, \rho}$ and $J_{n, \rho}$ for $\rho\in[1/2, 1]$.

\begin{theorem}\label{thm1.3}
    Let $\rho\in[1/2, 1]$ and $n\in\mathbb{N}$. Then we have
    \begin{align*}
        R_{n, \rho}=I_{n, 1/2}-I_{n, 1-\rho}\quad\text{and}\quad T_{n, \rho}=J_{n, 1/2}-J_{n, 1-\rho}.
    \end{align*}
\end{theorem}


\begin{corollary}\label{cor1.4}
    Let $n\in\mathbb{N}_{\ge 3}$ and $\rho\in[1/2,1]$.
    \begin{enumerate}
        \item If $n$ is odd, then
        \begin{align*}
            I_{n,\rho}&=\left(\frac{n-1}{n}\right)I_{n-2,\rho}+\displaystyle\frac{2(1-\rho)}{n}\sin^{n-1}((1-\rho)\pi)\cos((1-\rho)\pi)\\
                    &\qquad +\displaystyle\frac{4}{\pi n^2}(2-\sin^n((1-\rho)\pi)),\\
            J_{n,\rho}&=\left(\frac{n-1}{n}\right)J_{n-2, \rho}+\frac{(1-2\rho)}{n}\cos^{n-1}((1-\rho)\pi)\sin((1-\rho)\pi)\\
            &\qquad+\frac{4}{\pi n^2}(1+\cos^n((1-\rho)\pi)).
        \end{align*}
        \item If $n$ is even, then
        \begin{align*}
            I_{n,\rho}&=\left(\frac{n-1}{n}\right)I_{n-2,\rho}+\displaystyle\left(\frac{3-4\rho}{n}\right)\sin^{n-1}((1-\rho)\pi)\cos((1-\rho)\pi)\\
                    &\displaystyle\qquad+\frac{8}{\pi n^2}(2-\sin^n((1-\rho)\pi)),\\
        J_{n,\rho}&=\left(\frac{n-1}{n}\right)J_{n-2, \rho}+\left(\frac{3-4\rho}{n}\right)\cos^{n-1}((1-\rho)\pi)\sin((1-\rho)\pi)\\
            &\qquad+\frac{8}{\pi n^2}(1+\cos^n((1-\rho)\pi)).
        \end{align*}
    \end{enumerate}
\end{corollary}


\subsection{$I_{n, \rho}$, $J_{n, \rho}$ for general $\rho$}

Finally, we can compute $I_{n, \rho}$ and $J_{n, \rho}$ for general $\rho\in\mathbb{R}_{\ge 0}$ via the following formulas.

\begin{theorem}\label{thm1.5}
    Let $\rho\in\mathbb{R}_{\ge 0}$ and $n\in\mathbb{N}$. Then we have
    \begin{align*}
        I_{n, \rho}&=\begin{cases}
            \lfloor \rho\rfloor I_{n, 1}+I_{n, \{\rho\}} &\text{if $\lfloor \rho\rfloor$ is even},\\\\
            (\lfloor \rho\rfloor+1)I_{n, 1}-I_{n, 1-\{\rho\}} &\text{if $\lfloor \rho\rfloor$ is odd}.
        \end{cases}\\
        J_{n, \rho}&=\begin{cases}
           \lfloor \rho\rfloor J_{n, 1}+J_{n, \{\rho\}}   &\text{if $\lfloor \rho\rfloor$ is even},\\\\
            (\lfloor\rho\rfloor+1)J_{n, 1}-J_{n, 1-\{\rho\}}  &\text{if $\lfloor \rho\rfloor$ is odd}.
        \end{cases}
    \end{align*}
    Here $\{\rho\}:=\rho-\lfloor\rho\rfloor$ denotes the fractional part of $\rho$.
\end{theorem}


\section{Preliminaries}\label{sec2}

This section is dedicated to stating various useful lemmas used in our work. 

\subsection{Trigonometric identities and analysis facts}

First, we need an analogous result to Lemma 8 of \cite{MG}, but in a more general context. Recall the meaning of the Landau's symbol \cite{KEA}. Let $f$ and $g$ be two functions defined on some $I\subset \mathbb{R}$.  Then we write $f(x)=\mathcal{O}(g(x))$ if there exists $C>0$ such that $|f(x)|\le C|g(x)|$ for all $x\in I$.

\begin{lemma}\label{lem2.1}
    Let $\beta\in\mathbb{R}\setminus \{0\}$. Suppose $f(x)$ and $g(x)$ are functions of the rate $1/(\beta x)+\mathcal{O}(x)$, when $x\to 0$. Then we have
    \[\lim_{k\to\infty}\left(\left(\frac{1}{k}-1\right)f\left(\frac{x}{k-1}\right)+\left(\frac{1}{k}+1\right)g\left(\frac{x}{k+1}\right)\right)=\frac{4}{\beta x}.\]
\end{lemma}
\noindent
{\bf Proof.} The proof is trivial by some algebraic manipulation.
\hfill $\square$

\begin{corollary}\label{cor2.2}
		Let $\alpha_1,\alpha_2,\beta,\gamma,x\in\mathbb{R}$ such that $\beta,x\neq0$. As $k\to\infty$, we have
		\begin{align*}
		    \left(\frac{1}{k}-1\right)\frac{\cos\left(\frac{\alpha_1x}{k-1}+\gamma\right)}{\sin\left(\frac{\beta x}{k-1}\right)}+\left(\frac{1}{k}+1\right)\frac{\cos\left(\frac{\alpha_2x}{k+1}+\gamma\right)}{\sin\left(\frac{\beta x}{k+1}\right)}\to\frac{4}{\beta x}\cos\gamma+\left(\frac{\alpha_1-\alpha_2}{\beta}\right)\sin\gamma,\\
		\left(\frac{1}{k}-1\right)\frac{\sin\left(\frac{\alpha_1x}{k-1}+\gamma\right)}{\sin\left(\frac{\beta x}{k-1}\right)}+\left(\frac{1}{k}+1\right)\frac{\sin\left(\frac{\alpha_2x}{k+1}+\gamma\right)}{\sin\left(\frac{\beta x}{k+1}\right)}\to\frac{4}{\beta x}\sin\gamma-\left(\frac{\alpha_1-\alpha_2}{\beta}\right)\cos\gamma.
		\end{align*}
	\end{corollary}
    \noindent
    {\bf Proof.} We will only prove the first limit. Using the standard angle addition formula, we first note that
		\begin{align*}
			A_k:&=\left(\frac{1}{k}-1\right)\frac{\cos\left(\frac{\alpha_1x}{k-1}+\gamma\right)}{\sin\left(\frac{\beta x}{k-1}\right)}+\left(\frac{1}{k}+1\right)\frac{\cos\left(\frac{\alpha_2x}{k+1}+\gamma\right)}{\sin\left(\frac{\beta x}{k+1}\right)}\\
			&=\cos\gamma\left(\left(\frac{1}{k}-1\right)\frac{\cos\left(\frac{\alpha_1x}{k-1}\right)}{\sin\left(\frac{\beta x}{k-1}\right)}+\left(\frac{1}{k}+1\right)\frac{\cos\left(\frac{\alpha_2x}{k+1}\right)}{\sin\left(\frac{\beta x}{k+1}\right)}\right)\\
			&\qquad-\sin\gamma\left(\left(\frac{1}{k}-1\right)\frac{\sin\left(\frac{\alpha_1x}{k-1}\right)}{\sin\left(\frac{\beta x}{k-1}\right)}+\left(\frac{1}{k}+1\right)\frac{\sin\left(\frac{\alpha_2x}{k+1}\right)}{\sin\left(\frac{\beta x}{k+1}\right)}\right).
		\end{align*}
        Next, we recall the Maclaurin series of $\cos(\alpha x)$ and $\sin(\beta x)$:
        \[\cos(\alpha x)=1-\frac{\alpha^2 x^2}{2!}+\frac{\alpha^4x^4}{4!}+\cdots\quad\text{and}\quad\sin(\beta x)=\beta x-\frac{\beta^3x^3}{3!}+\frac{\beta^5x^5}{5!}+\cdots.\]
        By using the long division algorithm, we have
        \[\frac{\cos(\alpha x)}{\sin(\beta x)}=\frac{1}{\beta x}-\frac{1}{\beta}\left(\frac{\alpha^2}{2!}-\frac{\beta^2}{3!}\right)x+\cdots=\frac{1}{\beta x}+\mathcal{O}(x),\]
        where we consider when $x$ is sufficiently near $0$.
        Thus applying Lemma \ref{lem2.1}, we yield
        \[\left(\frac{1}{k}-1\right)\frac{\cos\left(\frac{\alpha_1x}{k-1}\right)}{\sin\left(\frac{\beta x}{k-1}\right)}+\left(\frac{1}{k}+1\right)\frac{\cos\left(\frac{\alpha_2x}{k+1}\right)}{\sin\left(\frac{\beta x}{k+1}\right)}\to\frac{4}{\beta x},\]
		as $k\to\infty$.
        Next, using the well-known limit identity, we see that
        \[\frac{\sin\left(\frac{\alpha x}{k\pm1}\right)}{\sin\left(\frac{\beta x}{k\pm1}\right)}\to\frac{\alpha}{\beta},\]
        as $k\to\infty$, where $\beta\neq 0$. This implies
		\[A_k\to\cos\gamma\left(\frac{4}{\beta x}\right)-\sin\gamma\left(-\frac{\alpha_1}{\beta}+\frac{\alpha_2}{\beta}\right)=\frac{4}{\beta x}\cos\gamma+\left(\frac{\alpha_1-\alpha_2}{\beta}\right)\sin\gamma,\]
		as $k\to\infty$. So we are done. \hfill $\square$

\hfill

Additionally, for the ease of reference, we state the following result which can be found in Section 1.3 of \cite{II}.

\begin{lemma}\label{lem2.3}
    Let $\alpha,\theta\in[0,2\pi)$ with $\theta\ne 0,\pi$, and $n\in\mathbb{N}$. Then the following hold.
		\begin{enumerate}
			\item $\displaystyle\sum_{k=1}^n\cos(\alpha+k\theta)=\dfrac{\sin\left(\alpha+\left(n+\frac{1}{2}\right)\theta\right)-\sin\left(\alpha+\frac{\theta}{2}\right)}{2\sin\left(\frac{\theta}{2}\right)}$.
			\item $\displaystyle\sum_{k=1}^n\cos(\alpha+(2k-1)\theta)=\dfrac{\sin(\alpha+2n\theta)-\sin\alpha}{2\sin\theta}$.
			\item $\displaystyle\sum_{k=1}^n\sin(\alpha+k\theta)=\dfrac{\cos\left(\alpha+\frac{\theta}{2}\right)-\cos\left(\alpha+\left(n+\frac{1}{2}\right)\theta\right)}{2\sin\left(\frac{\theta}{2}\right)}$.
			\item $\displaystyle\sum_{k=1}^n\sin(\alpha+(2k-1)\theta)=\dfrac{\cos\alpha-\cos(\alpha+2n\theta)}{2\sin\theta}$.
			\item $\sin^n\theta=\begin{cases}
				\displaystyle\frac{(-1)^{\frac{n-1}{2}}}{2^{n-1}}\sum_{m=0}^{(n-1)/2}(-1)^m{n\choose m}\sin((n-2m)\theta)&\text{if }n\text{ is odd,}\\\\
				\displaystyle\frac{1}{2^n}{n\choose n/2}+\frac{(-1)^{n/2}}{2^{n-1}}\sum_{m=0}^{(n-2)/2}(-1)^m{n\choose m}\cos((n-2m)\theta)&\text{if }n\text{ is even.}
			\end{cases}$
			\item $\cos^n\theta=\begin{cases}
				\displaystyle\frac{1}{2^{n-1}}\sum_{m=0}^{(n-1)/2}{n\choose m}\cos((n-2m)\theta)&\text{if }n\text{ is odd,}\\\\
				\displaystyle\frac{1}{2^n}{n\choose n/2}+\frac{1}{2^{n-1}}\sum_{m=0}^{(n-2)/2}{n\choose m}\cos((n-2m)\theta)&\text{if }n\text{ is even.}
			\end{cases}$
		\end{enumerate}
\end{lemma}

Lastly, we state a typical fact in analysis.

\begin{lemma}\label{lem2.4}
Let $A_1, \ldots, A_n$ partition $\mathbb{N}$. Suppose $A_1, \ldots, A_m$ are finite sets, where $m<n$, and $A_{m+1}, \ldots, A_n$ are infinite sets. Let $\{x_{j}\}_{j\in\mathbb{N}}$ be a sequence of real numbers. If $\{x_j\}_{j\in A_{i}}$, where $i=m+1, \ldots, n$, all converge to some $x\in\mathbb{R}$, then $\{x_j\}_{j\in\mathbb{N}}$ also converges to $x$.
\end{lemma}
\noindent
{\bf Proof.} Let $\varepsilon>0$. Then for each $j=m+1, \ldots, n$, there is $N_j>0$ such that for all $k\in A_j\cap [N_j, \infty)$, we have $|x_k-x|<\varepsilon$. We may assume that $[N_j, \infty)\cap\bigcup_{i=1}^mA_i =\emptyset$. Take $N:=\max\{N_{m+1}, \ldots, N_n\}$. Then $\{x_k\}\to x$ by definition, i.e., for all $k\ge N$, $|x_k-x|<\varepsilon$ (from $\bigsqcup_{j=1}^nA_j=\mathbb{N}$).\hfill $\square$

\subsection{Key ingredients}

The next few lemmas will be the main tools in proving our main results. Before that, we define
\[f_{j, k}(x):=\sin(jk x)-\sin(jx) \quad\text{and}\quad g_{j, k}(x):=\cos(jkx)-\cos(jx),\]
where $j, k\in\mathbb{N}$. Since all of them can be proven in the same way, we will provide a completely full detailed proof of Lemma \ref{lem2.5} when $j$ is odd, while in the proofs of other results, we may skip certain routine steps.

\begin{lemma}\label{lem2.5}
Let $\rho\in(0, 1/2]\cap\mathbb{Q}$ and $r\in\{0, 1\}$. 
\begin{enumerate}
			\item If $j$ is odd, then
			\begin{align*}
			    &\lim_{k\to\infty}\sum_{l=1}^{\left\lfloor\frac{\rho(k-1)}{2}\right\rfloor+r}\left(\int_{\frac{2(l-1)\pi}{k-1}}^{\frac{(2l-1)\pi}{k+1}}f_{j,k}(x)dx-\int_{\frac{(2l-1)\pi}{k+1}}^{\frac{2l\pi}{k-1}}f_{j,k}(x)dx\right)\\
                &\qquad=\frac{4}{\pi j^2}\sin(\rho j\pi)-\frac{2\rho}{j}\cos(\rho j\pi).
			\end{align*}
			\item If $j$ is even, then
			\begin{align*}
			    &\lim_{k\to\infty}\sum_{l=1}^{\lfloor\rho(k-1)\rfloor+r}\left(\int_{\frac{(l-1)\pi}{k-1}}^{\frac{l\pi}{k+1}}g_{j,k}(x)dx-\int_{\frac{l\pi}{k+1}}^{\frac{l\pi}{k-1}}g_{j,k}(x)dx\right)\\
                &\qquad=\frac{4\rho-1}{j}\sin(\rho j\pi)+\frac{8}{\pi j^2}(\cos(\rho j\pi)-1).
			\end{align*}
		\end{enumerate}
\end{lemma}
\noindent
{\bf Proof of Lemma \ref{lem2.5} (odd $j$ case).} For $l\in\mathbb{N}$, by the fundamental theorem of calculus and the angle addition formula, we have
\begin{align*}
    &\int_{\frac{2(l-1)\pi}{k-1}}^{\frac{(2l-1)\pi}{k+1}}f_{j, k}(x)dx
    =\frac{1}{j}\left(\frac{1}{k}+1\right)\cos\left(\frac{j(2l-1)\pi}{k+1}\right)+\frac{1}{j}\left(\frac{1}{k}-1\right)\cos\left(\frac{2j(l-1)\pi}{k-1}\right),
\end{align*}
and analogously
\[\int_{\frac{(2l-1)\pi}{k+1}}^{\frac{2l\pi }{k-1}}f_{j, k}(x)dx=\frac{1}{j}\left(\frac{1}{k}+1\right)\cos\left(\frac{j(2l-1)\pi}{k+1}\right)+\frac{1}{j}\left(\frac{1}{k}-1\right)\cos\left(\frac{2jl\pi}{k-1}\right).\]
So we can write
\[\sum_{l=1}^{\left\lfloor\frac{\rho(k-1)}{2}\right\rfloor+r}\left(\int_{\frac{2(l-1)\pi}{k-1}}^{\frac{(2l-1)\pi}{k+1}}f_{j,k}(x)dx-\int_{\frac{(2l-1)\pi}{k+1}}^{\frac{2l\pi}{k-1}}f_{j,k}(x)dx\right)=B_k+C_k,\]
		where
		\[B_k:=\frac{2}{j}\left(\frac{1}{k}+1\right)\sum_{l=1}^{\left\lfloor\frac{\rho(k-1)}{2}\right\rfloor+r}\cos\left(\frac{j(2l-1)\pi}{k+1}\right)\]
		and
		\begin{align*}
			C_k&:=\frac{1}{j}\left(\frac{1}{k}-1\right)\sum_{l=1}^{\left\lfloor\frac{\rho(k-1)}{2}\right\rfloor+r}\left(\cos\left(\frac{2j(l-1)\pi}{k-1}\right)+\cos\left(\frac{2jl\pi}{k-1}\right)\right)\\
			&=\frac{2}{j}\left(\frac{1}{k}-1\right)\cos\left(\frac{j\pi}{k-1}\right)\sum_{l=1}^{\left\lfloor\frac{\rho(k-1)}{2}\right\rfloor+r}\cos\left(\frac{j(2l-1)\pi}{k-1}\right).
		\end{align*}
        Note that we apply the sum-to-product trigonometric formula in the equality above.
In the next step, we further compute $B_k$ and $C_k$ using Lemma \ref{lem2.3}. We have
\begin{align*}
			B_k&=\frac{2}{j}\left(\frac{1}{k}+1\right)\frac{\sin\left(2\left(\left\lfloor\frac{\rho(k-1)}{2}\right\rfloor+r\right)\left(\frac{j\pi}{k+1}\right)\right)}{2\sin(\frac{j\pi}{k+1})}
		=\frac{1}{j}\left(\frac{1}{k}+1\right)\frac{\sin\left(j\rho\pi-\frac{2(\rho+s_k-r)j\pi}{k+1}\right)}{\sin\left(\frac{j\pi}{k+1}\right)},
		\end{align*}
		where $s_k:=\rho(k-1)/2-\lfloor\rho(k-1)/2\rfloor$. We also have
        \begin{align*}
			C_k&=\frac{2}{j}\left(\frac{1}{k}-1\right)\cos\left(\frac{j\pi}{k-1}\right)\left(\frac{\sin\left(2\left(\left\lfloor\frac{\rho(k-1)}{2}\right\rfloor+r\right)\frac{j\pi}{k-1}\right)}{2\sin\left(\frac{j\pi}{k-1}\right)}\right)\\
			&=\frac{1}{j}\left(\frac{1}{k}-1\right)\frac{\sin\left(\rho j\pi-\frac{(2s_k-2r-1)j\pi}{k-1}\right)}{\sin\left(\frac{j\pi}{k-1}\right)}-\frac{1}{j}\left(\frac{1}{k}-1\right)\cos\left(\rho j\pi-\frac{2(s_k-r)j\pi}{k-1}\right).
		\end{align*}
        Notice that there are finitely many possible values of $s_k$ due to the rationality of $\rho$.
Thus using Corollary \ref{cor2.2} and Lemma \ref{lem2.4}, we yield
\begin{align*}
    &\sum_{l=1}^{\left\lfloor\frac{\rho(k-1)}{2}\right\rfloor+r}\left(\int_{\frac{2(l-1)\pi}{k-1}}^{\frac{(2l-1)\pi}{k+1}}f_{j,k}(x)dx-\int_{\frac{(2l-1)\pi}{k+1}}^{\frac{2l\pi}{k-1}}f_{j,k}(x)dx\right)\\
    &\qquad=\frac{1}{j}\left(\frac{1}{k}-1\right)\frac{\sin\left(\rho j\pi-\frac{(2s_k-2r-1)j\pi}{k-1}\right)}{\sin\left(\frac{j\pi}{k-1}\right)}+\frac{1}{j}\left(\frac{1}{k}+1\right)\frac{\sin\left(\rho j\pi-\frac{2(\rho+s_k-r)j\pi}{k+1}\right)}{\sin\left(\frac{j\pi}{k+1}\right)}\\
				&\qquad\qquad-\frac{1}{j}\left(\frac{1}{k}-1\right)\cos\left(\rho j\pi-\frac{2(s_k-r)j\pi}{k-1}\right)\\
    &\qquad\to\frac{1}{j}\left(\frac{4}{j\pi}\sin(\rho j\pi)-\left(\frac{-(2s_k-2r-1)j\pi+2(\rho+s_k-r)j\pi}{j\pi}\right)\cos(\rho j\pi)\right)+\frac{1}{j}\cos(\rho j\pi)
\end{align*}
			as $k\to\infty$. By further simplification, we are done.
            \hfill $\square$

\hfill

\noindent
{\bf Proof of Lemma \ref{lem2.5} (even $j$ case).} For $l\in\mathbb{N}$, the fundamental theorem of calculus and the angle addition formula tell us that
\begin{align*}
				\int_{\frac{(l-1)\pi}{k-1}}^{\frac{l\pi}{k+1}}g_{j,k}(x)dx
				&=-\frac{1}{j}\left(\frac{1}{k}+1\right)\sin\left(\frac{jl\pi}{k+1}\right)-\frac{1}{j}\left(\frac{1}{k}-1\right)\sin\left(\frac{j(l-1)\pi}{k-1}\right)
			\end{align*}
and
    \begin{align*}
				\int_{\frac{l\pi}{k+1}}^{\frac{l\pi}{k-1}}g_{j,k}(x)dx
				=\frac{1}{j}\left(\frac{1}{k}-1\right)\sin\left(\frac{jl\pi}{k-1}\right)+\frac{1}{j}\left(\frac{1}{k}+1\right)\sin\left(\frac{jl\pi}{k+1}\right).
			\end{align*}
Hence we can write
			\[\sum_{l=1}^{\lfloor\rho(k-1)\rfloor+r}\left(\int_{\frac{(l-1)\pi}{k-1}}^{\frac{l\pi}{k+1}}g_{j,k}(x)dx-\int_{\frac{l\pi}{k+1}}^{\frac{l\pi}{k-1}}g_{j,k}(x)dx\right)=F_k+G_k,\]
			where
			\[F_k:=-\frac{2}{j}\left(\frac{1}{k}+1\right)\sum_{l=1}^{\lfloor\rho(k-1)\rfloor+r}\sin\left(\frac{jl\pi}{k+1}\right)\]
			and
			\begin{align*}
				G_k&:=-\frac{1}{j}\left(\frac{1}{k}-1\right)\sum_{l=1}^{\lfloor\rho(k-1)\rfloor+r}\left(\sin\left(\frac{j(l-1)\pi}{k-1}\right)+\sin\left(\frac{jl\pi}{k-1}\right)\right)\\
				&=-\frac{2}{j}\left(\frac{1}{k}-1\right)\cos\left(\frac{j\pi}{2(k-1)}\right)\sum_{l=1}^{\lfloor\rho(k-1)\rfloor+r}\sin\left(\frac{j(2l-1)\pi}{2(k-1)}\right).
			\end{align*}
So by applying Lemma \ref{lem2.3}, we may simplify $F_k$ and $G_k$ to
\begin{align*}
				F_k&=-\frac{2}{j}\left(\frac{1}{k}+1\right)\left(\frac{\cos\left(\frac{j\pi}{2(k+1)}\right)-\cos\left(\left(\lfloor\rho(k-1)\rfloor+r+\frac{1}{2}\right)\frac{j\pi}{k+1}\right)}{2\sin\left(\frac{j\pi}{2(k+1)}\right)}\right)\\
				&=-\frac{1}{j}\left(\frac{1}{k}+1\right)\left(\cot\left(\frac{j\pi}{2(k+1)}\right)-\frac{\cos\left(\rho j\pi-\frac{\left(4\rho+2v_k-2r-1\right)j\pi}{2(k+1)}\right)}{\sin\left(\frac{j\pi}{2(k+1)}\right)}\right),\\
				G_k&=-\frac{2}{j}\left(\frac{1}{k}-1\right)\cos\left(\frac{j\pi}{2(k-1)}\right)\left(\frac{1-\cos\left((\lfloor\rho(k-1)\rfloor+r)\frac{j\pi}{k-1}\right)}{2\sin\left(\frac{j\pi}{2(k-1)}\right)}\right)\\
				&=-\frac{1}{j}\left(\frac{1}{k}-1\right)\left(\cot\left(\frac{j\pi}{2(k-1)}\right)-\frac{\cos\left(\rho j\pi-\frac{(2v_k-2r-1)j\pi}{2(k-1)}\right)}{\sin\left(\frac{j\pi}{2(k-1)}\right)}\right)\\
				&\qquad+\frac{1}{j}\left(\frac{1}{k}-1\right)\sin\left(\rho j\pi-\frac{(v_k-r)j\pi}{k-1}\right).
			\end{align*}
Here, $v_k:=\rho(k-1)-\lfloor\rho(k-1)\rfloor$. Finally, by applying Corollary \ref{cor2.2} and Lemma \ref{lem2.4} under the $F_k$ and $G_k$ we have established, the proof is complete.\hfill $\square$

\begin{lemma}\label{newlem2.6}
    Let $\rho\in(0, 1/2]\cap\mathbb{Q}$ and $r\in\{0, 1\}$. If $j$ is odd, then
        \begin{align*}
            &\lim_{k\to\infty}\sum_{l=1}^{\left\lfloor \frac{\rho(k-1)}{2}\right\rfloor+r}\left(\int_{\frac{2l\pi}{k+1}}^{\frac{2l\pi}{k-1}}g_{j, k}(x)dx-\int_{\frac{2(l-1)\pi}{k-1}}^{\frac{2l\pi}{k+1}}g_{j, k}(x)dx\right)\\
            &\qquad=\frac{4}{\pi j^2}(1-\cos(\rho j\pi))-\frac{(2\rho-1)}{j}\sin(\rho j\pi).
        \end{align*}
\end{lemma}

\noindent
{\bf Proof.} Firstly, for $l\in\mathbb{N}$, using the fundamental theorem of calculus and the angle addition formula, we yield
\begin{align*}
    \int_{\frac{2l\pi}{k+1}}^{\frac{2l\pi}{k-1}}g_{j, k}(x)dx=\frac{1}{j}\left(\frac{1}{k}-1\right)\sin\left(\frac{2jl\pi}{k-1}\right)+\frac{1}{j}\left(\frac{1}{k}+1\right)\sin\left(\frac{2jl\pi}{k+1}\right)
\end{align*}
and
\begin{align*}
    \int_{\frac{2(l-1)\pi}{k-1}}^{\frac{2l\pi}{k+1}}g_{j, k}(x)dx=-\frac{1}{j}\left(\frac{1}{k}+1\right)\sin\left(\frac{2jl\pi}{k+1}\right)-\frac{1}{j}\left(\frac{1}{k}-1\right)\sin\left(\frac{2j(l-1)\pi}{k-1}\right).
\end{align*}
Then we can write
\begin{align*}
    \sum_{l=1}^{\left\lfloor \frac{\rho(k-1)}{2}\right\rfloor+r}\left(\int_{\frac{2l\pi}{k+1}}^{\frac{2l\pi}{k-1}}g_{j, k}(x)dx-\int_{\frac{2(l-1)\pi}{k-1}}^{\frac{2l\pi}{k+1}}g_{j, k}(x)dx\right)=B_k+C_k,
\end{align*}
where
    \[B_k:=\frac{2}{j}\left(\frac{1}{k}+1\right)\sum_{l=1}^{\left\lfloor\frac{\rho(k-1)}{2}\right\rfloor+r}\sin\left(\frac{2jl\pi}{k+1}\right)\]
and
    \begin{align*}
        C_k&:=\frac{1}{j}\left(\frac{1}{k}-1\right)\sum_{l=1}^{\left\lfloor\frac{\rho(k-1)}{2}\right\rfloor+r}\left(\sin\left(\frac{2jl\pi}{k-1}\right)+\sin\left(\frac{2j(l-1)\pi}{k-1}\right)\right)\\
        &=\frac{2}{j}\left(\frac{1}{k}-1\right)\cos\left(\frac{j\pi}{k-1}\right)\sum_{l=1}^{\left\lfloor\frac{\rho(k-1)}{2}\right\rfloor+r}\sin\left(\frac{j(2l-1)\pi}{k-1}\right).
    \end{align*}
Next, we can simplify $B_k$ and $C_k$ using Lemma \ref{lem2.3} as follows.
\begin{align*}
    B_k&=\frac{2}{j}\left(\frac{1}{k}+1\right)\left(\frac{\cos\left(\frac{j\pi}{k+1}\right)-\cos\left(\left(\left\lfloor \frac{\rho(k-1)}{2}\right\rfloor+r+\frac{1}{2}\right)\frac{2j\pi}{k+1}\right)}{2\sin\left(\frac{j\pi}{k+1}\right)}\right)\\
    &=\frac{1}{j}\left(\frac{1}{k}+1\right)\left(\cot\left(\frac{j\pi}{k+1}\right)-\frac{\cos\left(\rho j\pi-\frac{(2\rho+2s_k-2r-1)j\pi}{k+1}\right)}{\sin\left(\frac{j\pi}{k+1}\right)}\right).\\
    C_k&=\frac{2}{j}\left(\frac{1}{k}-1\right)\cos\left(\frac{j\pi}{k-1}\right)\left(\frac{1-\cos\left(\left(\left\lfloor\frac{\rho(k-1)}{2}\right\rfloor+r\right)\frac{2j\pi}{k-1}\right)}{2\sin\left(\frac{j\pi}{k-1}\right)}\right)\\
    &=\frac{1}{j}\left(\frac{1}{k}-1\right)\left(\cot\left(\frac{j\pi}{k-1}\right)-\frac{\cos\left(\rho j\pi-\frac{(2s_k-2r-1)j\pi}{k-1}\right)}{\sin\left(\frac{j\pi}{k-1}\right)}\right)\\
    &\qquad-\frac{1}{j}\left(\frac{1}{k}-1\right)\sin\left(\rho j\pi-\frac{(2s_k-2r)j\pi}{k-1}\right).
\end{align*}
Here, $s_k=\rho(k-1)/2-\lfloor\rho(k-1)/2\rfloor$. Lastly, by using Corollary \ref{cor2.2} and Lemma \ref{lem2.4}, we yield the desired result. \hfill $\square$


\begin{lemma}\label{lem2.6}
Let $\rho\in(1/2, 1]\cap\mathbb{Q}$, $j$ be an odd number, and $r\in\{0, 1\}$.  Then
\begin{align*}
		&\lim_{k\to\infty}\sum_{l=1}^{\left\lfloor\frac{\rho(k+1)+1}{2}\right\rfloor-\left\lceil\frac{k+3}{4}\right\rceil+r}\left(\int_{\left(2l+2\left\lceil\frac{k+3}{4}\right\rceil-3\right)\frac{\pi}{k+1}}^{(2l+2\left\lceil\frac{k+3}{4}\right\rceil-4)\frac{\pi}{k-1}}f_{j,k}(x)dx-\int_{\left(2l+2\left\lceil\frac{k+3}{4}\right\rceil-4\right)\frac{\pi}{k-1}}^{\left(2l+2\left\lceil\frac{k+3}{4}\right\rceil-1\right)\frac{\pi}{k+1}}f_{j,k}(x)dx\right)\\
				&\qquad=\frac{4}{\pi j^2}\left((-1)^{\frac{j-1}{2}}-\sin(\rho j\pi)\right)-\frac{2(1-\rho)}{j}\cos(\rho j\pi)
                \end{align*}
    and
    \begin{align*}
		&\lim_{k\to\infty}\sum_{l=1}^{\left\lfloor\frac{\rho(k+1)}{2}\right\rfloor-\left\lceil\frac{k+1}{4}\right\rceil+r}\left(\int_{\left(2l+2\left\lceil\frac{k+1}{4}\right\rceil-2\right)\frac{\pi}{k+1}}^{(2l+2\left\lceil\frac{k+1}{4}\right\rceil-2)\frac{\pi}{k-1}}g_{j,k}(x)dx-\int_{\left(2l+2\left\lceil\frac{k+1}{4}\right\rceil-2\right)\frac{\pi}{k-1}}^{\left(2l+2\left\lceil\frac{k+1}{4}\right\rceil\right)\frac{\pi}{k+1}}g_{j,k}(x)dx\right)\\
				&\qquad=-\frac{4}{\pi j^2}\cos(\rho j\pi)+\frac{(1-2\rho)}{j}\sin(\rho j\pi).
                \end{align*}
\end{lemma}
\noindent
{\bf Proof of Lemma \ref{lem2.6} (first equation).} 
 By the fundamental theorem of calculus and the angle addition formula, we first have
\begin{align*}
				&\int_{\left(2\left\lceil\frac{k+3}{4}\right\rceil+2l-3\right)\frac{\pi}{k+1}}^{\left(2\left\lceil\frac{k+3}{4}\right\rceil+2l-4\right)\frac{\pi}{k-1}}f_{j,k}(x)dx\\
				&\qquad=-\frac{1}{j}\left(\frac{1}{k}-1\right)\cos\left(\frac{j\pi}{k-1}\left(2\left\lceil\frac{k+3}{4}\right\rceil+2l-4\right)\right)\\
				&\qquad\qquad-\frac{1}{j}\left(\frac{1}{k}+1\right)\cos\left(\frac{j\pi}{k+1}\left(2\left\lceil\frac{k+3}{4}\right\rceil+2l-3\right)\right)
			\end{align*}
and 
\begin{align*}
    &\int_{\left(2\left\lceil\frac{k+3}{4}\right\rceil+2l-4\right)\frac{\pi}{k-1}}^{\left(2\left\lceil\frac{k+3}{4}\right\rceil+2l-1\right)\frac{\pi}{k+1}}f_{j,k}(x)dx\\
    &\qquad=\frac{1}{j}\left(\frac{1}{k}+1\right)\cos\left(\frac{j\pi}{k+1}\left(2\left\lceil\frac{k+3}{4}\right\rceil+2l-1\right)\right)\\
				&\qquad\qquad+\frac{1}{j}\left(\frac{1}{k}-1\right)\cos\left(\frac{j\pi}{k-1}\left(2\left\lceil\frac{k+3}{4}\right\rceil+2l-4\right)\right).
\end{align*}
So we can write
\begin{align}\label{oldDE}
    \sum_{l=1}^{\alpha}\left(\int_{\left(2\left\lceil\frac{k+3}{4}\right\rceil+2l-3\right)\frac{\pi}{k+1}}^{(2\left\lceil\frac{k+3}{4}\right\rceil+2l-4)\frac{\pi}{k-1}}f_{j,k}(x)dx-\int_{\left(2\left\lceil\frac{k+3}{4}\right\rceil+2l-4\right)\frac{\pi}{k-1}}^{\left(2\left\lceil\frac{k+3}{4}\right\rceil+2l-1\right)\frac{\pi}{k+1}}f_{j,k}(x)dx\right)=D_k+E_k,
\end{align}
where
\begin{align*}
D_k&:=-\frac{2}{j}\left(\frac{1}{k}-1\right)\sum_{l=1}^{\alpha}\cos\left(\frac{j\pi}{k-1}\left(2\left\lceil\frac{k+3}{4}\right\rceil+2l-4\right)\right),\\
				E_k&:=-\frac{2}{j}\left(\frac{1}{k}+1\right)\cos\left(\frac{j\pi}{k+1}\right)\sum_{l=1}^{\alpha}\cos\left(\frac{j\pi}{k+1}\left(2\left\lceil\frac{k+3}{4}\right\rceil+2l-2\right)\right),
			\end{align*}
            \[\alpha:=\left\lfloor\frac{\rho(k+1)+1}{2}\right\rfloor-\left\lceil\frac{k+3}{4}\right\rceil+r=\frac{(2\rho-1)(k+1)}{4}-u_k-t_k+r.\]
Here, we define
\[u_k:=\frac{\rho(k+1)+1}{2}-\left\lfloor\frac{\rho(k+1)+1}{2}\right\rfloor\quad\text{and}\quad t_k:=\left\lceil\frac{k+3}{4}\right\rceil-\frac{k+3}{4}.\]
Using Lemma \ref{lem2.3} and some algebraic manipulation, we see that
\begin{align*}
				&\sum_{l=1}^\alpha\cos\left(\frac{j\pi}{k-1}\left(2\left\lceil\frac{k+3}{4}\right\rceil+2l-4\right)\right)\\
				&\qquad=\sum_{l=1}^\alpha\cos\left(\frac{j\pi}{k-1}\left(2\left\lceil\frac{k+3}{4}\right\rceil-3\right)+(2l-1)\frac{j\pi}{k-1}\right)\\
				&\qquad=\frac{\sin\left(\frac{j\pi}{k-1}\left(2\left\lceil\frac{k+3}{4}\right\rceil-3\right)+\frac{2\alpha j\pi}{k-1}\right)-\sin\left(\frac{j\pi}{k-1}\left(2\left\lceil\frac{k+3}{4}\right\rceil-3\right)\right)}{2\sin\left(\frac{j\pi}{k-1}\right)}\\
				&\qquad=\frac{\sin\left(\rho j\pi+\frac{(2\rho-2u_k+2r-2)j\pi}{k-1}\right)-(-1)^{\frac{j-1}{2}}\cos\left(\frac{(2t_k-1)j\pi}{k-1}\right)}{2\sin\left(\frac{j\pi}{k-1}\right)},
			\end{align*}
and similarly
\begin{align*}
				&\sum_{l=1}^\alpha\cos\left(\frac{j\pi}{k+1}\left(2\left\lceil\frac{k+3}{4}\right\rceil+2l-2\right)\right)\\
                &\qquad=\frac{\sin\left(\rho j\pi+\frac{(-2u_k+2r)j\pi}{k+1}\right)-(-1)^{\frac{j-1}{2}}\cos\left(\frac{2t_kj\pi}{k+1}\right)}{2\sin\left(\frac{j\pi}{k+1}\right)}.
			\end{align*}
Hence we can simplify $D_k$ and $E_k$ as follows.

\begin{align*}
D_k&=-\frac{1}{j}\left(\frac{1}{k}-1\right)\left(\frac{\sin\left(\rho j\pi+\frac{(2\rho-2u_k+2r-2)j\pi}{k-1}\right)-(-1)^{\frac{j-1}{2}}\cos\left(\frac{(2t_k-1)j\pi}{k-1}\right)}{\sin\left(\frac{j\pi}{k-1}\right)}\right),\\
				E_k
				&=-\frac{1}{j}\left(\frac{1}{k}+1\right)\left(\frac{\sin\left(\rho j\pi+\frac{(-2u_k+2r-1)j\pi}{k+1}\right)-(-1)^{\frac{j-1}{2}}\cos\left(\frac{(2t_k+1)j\pi}{k+1}\right)}{\sin\left(\frac{j\pi}{k+1}\right)}\right)\\
				&\qquad-\frac{1}{j}\left(\frac{1}{k}+1\right)\left(\cos\left(\rho j\pi+\frac{(-2u_k+2r)j\pi}{k+1}\right)-(-1)^{\frac{j-1}{2}}\sin\left(\frac{2t_kj\pi}{k+1}\right)\right).
			\end{align*}
Thus, applying Corollary \ref{cor2.2} and Lemma \ref{lem2.4} to \eqref{oldDE} on the basis of $D_k$ and $E_k$ we have found, the proof is complete.\hfill $\square$

\hfill

\noindent
{\bf Proof of Lemma \ref{lem2.6} (second equation).} First, we have 
\begin{align*}
				&\int_{\left(2l+2\left\lceil\frac{k+1}{4}\right\rceil-2\right)\frac{\pi}{k+1}}^{\left(2l+2\left\lceil\frac{k+1}{4}\right\rceil-2\right)\frac{\pi}{k-1}}g_{j,k}(x)dx\\
				&\qquad=\frac{1}{j}\left(\frac{1}{k}-1\right)\sin\left(\frac{j\pi}{k-1}\left(2\left\lceil\frac{k+1}{4}\right\rceil+2l-2\right)\right)\\
				&\qquad\qquad+\frac{1}{j}\left(\frac{1}{k}+1\right)\sin\left(\frac{j\pi}{k+1}\left(2l+2\left\lceil\frac{k+1}{4}\right\rceil-2\right)\right)
			\end{align*}
and 
\begin{align*}
    &\int_{\left(2l+2\left\lceil\frac{k+1}{4}\right\rceil-2\right)\frac{\pi}{k-1}}^{\left(2l+2\left\lceil\frac{k+1}{4}\right\rceil\right)\frac{\pi}{k+1}}g_{j,k}(x)dx\\
    &\qquad=-\frac{1}{j}\left(\frac{1}{k}+1\right)\sin\left(\frac{j\pi}{k+1}\left(2l+2\left\lceil\frac{k+1}{4}\right\rceil\right)\right)\\
				&\qquad\qquad-\frac{1}{j}\left(\frac{1}{k}-1\right)\sin\left(\frac{j\pi}{k-1}\left(2l+2\left\lceil\frac{k+1}{4}\right\rceil-2\right)\right).
\end{align*}
Thus we can write
\begin{align}\label{DE}
    \sum_{l=1}^{\beta}\left(\int_{\left(2l+2\left\lceil\frac{k+1}{4}\right\rceil-2\right)\frac{\pi}{k+1}}^{(2l+2\left\lceil\frac{k+1}{4}\right\rceil-2)\frac{\pi}{k-1}}g_{j,k}(x)dx-\int_{\left(2l+2\left\lceil\frac{k+1}{4}\right\rceil-2\right)\frac{\pi}{k-1}}^{\left(2l+2\left\lceil\frac{k+1}{4}\right\rceil\right)\frac{\pi}{k+1}}g_{j,k}(x)dx\right)=D_k+E_k,
\end{align}
where
\begin{align*}
D_k&:=\frac{2}{j}\left(\frac{1}{k}-1\right)\sum_{l=1}^{\beta}\sin\left(\frac{j\pi}{k-1}\left(2l+2\left\lceil\frac{k+1}{4}\right\rceil-2\right)\right),\\
				E_k&:=\frac{2}{j}\left(\frac{1}{k}+1\right)\cos\left(\frac{j\pi}{k+1}\right)\sum_{l=1}^{\beta}\sin\left(\frac{j\pi}{k+1}\left(2l+2\left\lceil\frac{k+1}{4}\right\rceil-1\right)\right),
			\end{align*}
            \[\beta:=\left\lfloor\frac{\rho(k+1)}{2}\right\rfloor-\left\lceil\frac{k+1}{4}\right\rceil+r=\frac{(2\rho-1)(k+1)}{4}-u_k-t_k+r.\]
Here, we define
\[u_k:=\frac{\rho(k+1)}{2}-\left\lfloor\frac{\rho(k+1)}{2}\right\rfloor\quad\text{and}\quad t_k:=\left\lceil\frac{k+1}{4}\right\rceil-\frac{k+1}{4}.\]
Using Lemma \ref{lem2.3} and some algebraic manipulation, we have
\begin{align*}
				&\sum_{l=1}^\beta\sin\left(\frac{j\pi}{k-1}\left(2l+2\left\lceil\frac{k+1}{4}\right\rceil-2\right)\right)\\
                &\qquad=\frac{(-1)^{\frac{j+1}{2}}\sin\left(\frac{2t_kj\pi}{k-1}\right)-\cos\left(\rho j\pi+\frac{(2\rho-2u_k+2r-1)j\pi}{k-1}\right)}{2\sin\left(\frac{j\pi}{k-1}\right)}
			\end{align*}
and 
\begin{align*}
				&\sum_{l=1}^\beta\sin\left(\frac{j\pi}{k+1}\left(2l+2\left\lceil\frac{k+1}{4}\right\rceil-1\right)\right)\\
                &\qquad=\frac{(-1)^{\frac{j+1}{2}}\sin\left(\frac{2t_kj\pi}{k+1}\right)-\cos\left(\rho j\pi-\frac{(2u_k-2r)j\pi}{k+1}\right)}{2\sin\left(\frac{j\pi}{k+1}\right)}.
			\end{align*}
Hence we can simplify $D_k$ and $E_k$ as follows.  

\begin{align*}
    D_k&=\frac{1}{j}\left(\frac{1}{k}-1\right)\left(\frac{(-1)^{\frac{j+1}{2}}\sin\left(\frac{2t_kj\pi}{k-1}\right)-\cos\left(\rho j\pi+\frac{(2\rho-2u_k+2r-1)j\pi}{k-1}\right)}{\sin\left(\frac{j\pi}{k-1}\right)}\right),\\
				E_k
				&=\frac{1}{j}\left(\frac{1}{k}+1\right)\left(\frac{(-1)^{\frac{j+1}{2}}\sin\left(\frac{(2t_k+1)j\pi}{k+1}\right)-\cos\left(\rho j\pi-\frac{(2u_k-2r-1)j\pi}{k+1}\right)}{\sin\left(\frac{j\pi}{k+1}\right)}\right)\\
				&\qquad-\frac{1}{j}\left(\frac{1}{k}+1\right)\left((-1)^{\frac{j+1}{2}}\cos\left(\frac{2t_kj\pi}{k+1}\right)+\sin\left(\rho j\pi-\frac{(2u_k-2r)j\pi}{k+1}\right)\right).
			\end{align*}
Therefore, using Corollary \ref{cor2.2} and Lemma \ref{lem2.4} to \eqref{DE} on the basis of $D_k$ and $E_k$ we have found, the proof is complete.\hfill $\square$

\hfill

Lastly, we mention here that from Section \ref{sec3} to Section \ref{sec6}, we will prove the results for when $\rho\in\mathbb{Q}$. For the case $\rho\in\mathbb{R}\setminus\mathbb{Q}$, we will discuss it in Section \ref{sec7}.

\section{Proof of results part I}\label{sec3}

Before proving Theorem \ref{thm1.1}, for each $n, m, k\in\mathbb{N}\cup\{0\}$, define
\[f_{n,m,k}(x):=\sin((n-2m)kx)-\sin((n-2m)x),\]
		\[g_{n,m,k}(x):=\cos((n-2m)kx)-\cos((n-2m)x).\]
If $\rho=0$, then Theorem \ref{thm1.1} clearly holds. So we may assume $\rho>0$ and $k\ge (2/\rho)+1$.

Let us briefly outline the main idea of the proof. The technique is standard: we first determine the \(x\)-intercepts of \(s_{n,k}\), which allows us to identify the regions where the function is positive or negative. This, in turn, permits us to remove the absolute value. Then let $k\to\infty$. We can express \(I_{n,\rho}\) as a sum of elementary integrals (which coincide with results in Section \ref{sec2}), together with a remainder term, which tends to \(0\). We now present the details of the proof.\\

\noindent
{\bf Proof of Theorem \ref{thm1.1} ($I_{n, \rho}$, odd $n$ case).} We set $\alpha_k:=\lfloor \rho(k-1)/2\rfloor$.  Note that $\alpha_k\ge 1$. We define the following sets:
\begin{align*}
    M_{\rho}&:=\left\{k\in\mathbb{N} : \alpha_k\in\left(\frac{\rho(k-1)}{2}-1,\frac{\rho(k+1)-1}{2}\right]\right\},\\
		N_{\rho}&:=\left\{k\in\mathbb{N} : \alpha_k\in\left(\frac{\rho(k+1)-1}{2},\frac{\rho(k-1)}{2}\right]\right\}.
\end{align*}
It is clear that $M_{\rho}$ and $N_{\rho}$ partition $\mathbb{N}$. Next, we find the intersections of the curves $\sin^n(kx)$ and $\sin^nx$ over $[0, \rho\pi]$ by solving the equation
\[\sin^n(kx)=\sin^nx.\]
Since $n$ is odd, taking the $n$th root on both sides gives $\sin(kx)=\sin x$. Then the sum-to-product formula implies
\[2\cos\left(\frac{(k+1)x}{2}\right)\sin\left(\frac{(k-1)x}{2}\right)=0.\]
So all the intersection points of these two functions on $[0, \rho\pi]$ are
\[x=\frac{(2a-1)\pi}{k+1}\quad\text{and}\quad x=\frac{2b\pi}{k-1},\]
where $a, b\in\mathbb{Z}$, and
\begin{align*}
    1\le a\le \left\lfloor\frac{\rho(k+1)+1}{2}\right\rfloor,\quad
    0\le b\le\left\lfloor\frac{\rho(k-1)}{2}\right\rfloor=\alpha_k.
\end{align*}
\begin{itemize}
    \item If $k\in M_\rho$, then
		\[\frac{2\alpha_k\pi}{k-1}<\frac{(2\alpha_k+1)\pi}{k+1}\leq\rho\pi<\frac{(2\alpha_k+2)\pi}{k-1}\quad\text{and}\quad\frac{2(l-1)\pi}{k-1}<\frac{(2l-1)\pi}{k+1}<\frac{2l\pi}{k-1}\]
		for $1\leq l\leq\alpha_k$, and
		\[s_{n,k}(x)\begin{cases}
			>0&\text{if }x\in\displaystyle\bigsqcup_{l=1}^{\alpha_k}\left(\dfrac{2(l-1)\pi}{k-1},\dfrac{(2l-1)\pi}{k+1}\right)\sqcup\left(\dfrac{2\alpha_k\pi}{k-1},\frac{(2\alpha_k+1)\pi}{k+1}\right),\\
			&\\
			<0&\text{if }x\in\displaystyle\bigsqcup_{l=1}^{\alpha_k}\left(\dfrac{(2l-1)\pi}{k+1},\dfrac{2l\pi}{k-1}\right)\sqcup\left(\frac{(2\alpha_k+1)\pi}{k+1},\rho\pi\right).
		\end{cases}\]
		That is,
		\begin{align*}
			I_{n,k,\rho}&=\sum_{l=1}^{\alpha_k}\left(\int_{\frac{2(l-1)\pi}{k-1}}^{\frac{(2l-1)\pi}{k+1}}s_{n,k}(x)dx-\int_{\frac{(2l-1)\pi}{k+1}}^{\frac{2l\pi}{k-1}}s_{n,k}(x)dx\right)\\
			&\qquad+\int_{\frac{2\alpha_k\pi}{k-1}}^{\frac{(2\alpha_k+1)\pi}{k+1}}s_{n,k}(x)dx-\int_{\frac{(2\alpha_k+1)\pi}{k+1}}^{\rho\pi}s_{n,k}(x)dx\\
			&=\sum_{l=1}^{\alpha_k+1}\left(\int_{\frac{2(l-1)\pi}{k-1}}^{\frac{(2l-1)\pi}{k+1}}s_{n,k}(x)dx-\int_{\frac{(2l-1)\pi}{k+1}}^{\frac{2l\pi}{k-1}}s_{n,k}(x)dx\right)+\int_{\rho\pi}^{\frac{2(\alpha_k+1)\pi}{k-1}}s_{n,k}(x)dx.
		\end{align*}
    \item If $k\in N_{\rho}$, then
		\[\frac{2\alpha_k\pi}{k-1}\leq\rho\pi<\frac{(2\alpha_k+1)\pi}{k+1},\quad\text{and}\quad\frac{2(l-1)\pi}{k-1}<\frac{(2l-1)\pi}{k+1}<\frac{2l\pi}{k-1}\]
		for $1\leq l\leq\alpha_k$, and
		\[s_{n,k}(x)\begin{cases}
			>0&\text{if }x\in\displaystyle\bigsqcup_{l=1}^{\alpha_k}\left(\dfrac{2(l-1)\pi}{k-1},\dfrac{(2l-1)\pi}{k+1}\right)\sqcup\left(\dfrac{2\alpha_k\pi}{k-1},\rho\pi\right),\\
			&\\
			<0&\text{if }x\in\displaystyle\bigsqcup_{l=1}^{\alpha_k}\left(\dfrac{(2l-1)\pi}{k+1},\dfrac{2l\pi}{k-1}\right).
		\end{cases}\]
        Thus we have
		\[I_{n,k,\rho}=\sum_{l=1}^{\alpha_k}\left(\int_{\frac{2(l-1)\pi}{k-1}}^{\frac{(2l-1)\pi}{k+1}}s_{n,k}(x)dx-\int_{\frac{(2l-1)\pi}{k+1}}^{\frac{2l\pi}{k-1}}s_{n,k}(x)dx\right)+\int_{\frac{2\alpha_k\pi}{k-1}}^{\rho\pi}s_{n,k}(x)dx.\]
\end{itemize}
        By combining the two cases, we can write
		\begin{align}\label{sum}
		    I_{n,k,\rho}=\sum_{l=1}^{\beta_k}\left(\int_{\frac{2(l-1)\pi}{k-1}}^{\frac{(2l-1)\pi}{k+1}}s_{n,k}(x)dx-\int_{\frac{(2l-1)\pi}{k+1}}^{\frac{2l\pi}{k-1}}s_{n,k}(x)dx\right)+\omega_k,
		\end{align}
		where
		\[\beta_k:=\begin{cases}
			\alpha_k&\text{if }k\in N_{\rho},\\
			\alpha_k+1&\text{if }k\in M_{\rho},
		\end{cases}\quad\text{and}\quad\omega_k:=\begin{cases}
		\displaystyle\int_{\frac{2\beta_k\pi}{k-1}}^{\rho\pi}s_{n,k}(x)dx&\text{if }k\in N_{\rho},\\
		&\\
		\displaystyle\int_{\rho\pi}^{\frac{2\beta_k\pi}{k-1}}s_{n,k}(x)dx&\text{if }k\in M_{\rho}.
		\end{cases}\]
        Let us consider the first term on \eqref{sum}. By applying Lemma \ref{lem2.3}, Lemma \ref{lem2.4}, and Lemma \ref{lem2.5} (part 1), we have
        \begin{align*}
			&\sum_{l=1}^{\beta_k}\left(\int_{\frac{2(l-1)\pi}{k-1}}^{\frac{(2l-1)\pi}{k+1}}s_{n,k}(x)dx-\int_{\frac{(2l-1)\pi}{k+1}}^{\frac{2l\pi}{k-1}}s_{n,k}(x)dx\right)\\
			&\qquad=\frac{(-1)^{\frac{n-1}{2}}}{2^{n-1}}\sum_{m=0}^{(n-1)/2}(-1)^m{n\choose m}\sum_{l=1}^{\beta_k}\left(\int_{\frac{2(l-1)\pi}{k-1}}^{\frac{(2l-1)\pi}{k+1}}f_{n,m,k}(x)dx-\int_{\frac{(2l-1)\pi}{k+1}}^{\frac{2l\pi}{k-1}}f_{n,m,k}(x)dx\right)\\
			&\qquad\to\frac{(-1)^{\frac{n-1}{2}}}{2^{n-3}\pi}\sum_{m=0}^{(n-1)/2}(-1)^m{n\choose m}\frac{\sin((n-2m)\rho\pi)}{(n-2m)^2}\\
			&\qquad\qquad+(-1)^{\frac{n+1}{2}}\frac{\rho}{2^{n-2}}\sum_{m=0}^{(n-1)/2}(-1)^m{n\choose m}\frac{\cos((n-2m)\rho\pi)}{n-2m},
		\end{align*}
		as $k\to\infty$ and $k\in M_{\rho},N_{\rho}$, since $n-2m$ is odd for $m\in\mathbb{N}$.
For the term $\omega_k$ in \eqref{sum}, as
		\[\frac{\rho}{2}-\frac{1}{k-1}<\frac{\alpha_k}{k-1}\leq\frac{\rho}{2}\]
		for $k\in\mathbb{N}$, we have $\alpha_k/(k-1)\to\rho/2$ as $k\to\infty$ by the sandwich theorem. Hence $\beta_k/(k-1)\to\rho/2$ as $k\to\infty$.
		This implies 
		\[\lvert\omega_k\rvert\le 2\left\lvert\rho\pi-\frac{2\beta_k\pi}{k-1}\right\rvert\to0,\]
		as $k\to\infty$ and $k\in M_{\rho},N_{\rho}$.
		By Lemma \ref{lem2.4}, $I_{n,\rho}$ exists and equals what we want. \hfill $\square$

\hfill

\noindent
{\bf Proof of Theorem \ref{thm1.1} ($J_{n, \rho}$, odd $n$ case).} Define $\alpha_k=\lfloor \rho(k-1)/2\rfloor\ge 1$. We wish to solve $\cos^nx=\cos^n(kx)$. By taking the $n$th root and using the sum-to-product formula, we have
\begin{align*}
    \sin\left(\frac{(k+1)x}{2}\right)\sin\left(\frac{(k-1)x}{2}\right)=0.
\end{align*}
Its solutions on $[0, \rho\pi]$ are $x=2a\pi/(k+1)$ and $x=2b\pi/(k-1)$,
where $a, b\in\mathbb{Z}$ and
\begin{align*}
    1\le a\le \left\lfloor\frac{\rho(k+1)}{2}\right\rfloor,\quad
    0\le b\le\left\lfloor\frac{\rho(k-1)}{2}\right\rfloor=\alpha_k.
\end{align*}

Now, one can appropriately partition $\mathbb{N}=M_{\rho}\sqcup N_{\rho}$ as in the previous proof. We will omit the  details.
\begin{itemize}
    \item If $k\in M_{\rho}$, then we have
    \begin{align*}
        c_{n, k}(x)\begin{cases}
            >0 &\text{if }\displaystyle x\in\bigsqcup_{l=1}^{\alpha_k}\left(\frac{2l\pi}{k+1}, \frac{2l\pi}{k-1}\right)\sqcup\left(\frac{2(\alpha_k+1)\pi}{k+1}, \rho\pi\right),
            \\
            \\
            <0 &\text{if }\displaystyle x\in\bigsqcup_{l=1}^{\alpha_k+1}\left(\frac{2(l-1)\pi}{k-1}, \frac{2l\pi}{k+1}\right).
        \end{cases}
    \end{align*}
    \item If $k\in N_{\rho}$, then we have
    \begin{align*}
        c_{n, k}(x)\begin{cases}
            >0 &\text{if }\displaystyle x\in\bigsqcup_{l=1}^{\alpha_k}\left(\frac{2l\pi}{k+1}, \frac{2l\pi}{k-1}\right),
            \\
            \\
            <0 &\text{if }\displaystyle x\in\bigsqcup_{l=1}^{\alpha_k}\left(\frac{2(l-1)\pi}{k-1}, \frac{2l\pi}{k+1}\right)\sqcup\left(\frac{2\alpha_k\pi}{k-1}, \rho\pi\right).
            \end{cases}
    \end{align*}
\end{itemize}
Combining these two cases, we deduce
\begin{align}\label{neweq}
    J_{n, k, \rho}&=\sum_{l=1}^{\lambda_k}\left(\int_{\frac{2l\pi}{k+1}}^{\frac{2l\pi}{k-1}}c_{n, k}(x)dx-\int_{\frac{2(l-1)\pi}{k-1}}^{\frac{2l\pi}{k+1}} c_{n, k}(x)dx\right)+\tau_k,
\end{align}
where
\begin{align*}
    \lambda_k:=\begin{cases}
        \alpha_{k}+1 &\text{if } k\in M_{\rho},\\
        \alpha_k &\text{if } k\in N_{\rho},
    \end{cases}\quad\text{and}\quad \tau_{k}:=\begin{cases}
        \displaystyle\int_{\frac{2(\alpha_k+1)\pi}{k-1}}^{\rho\pi}c_{n, k}(x)dx &\text{if } k\in M_{\rho},\\
        \\
        \displaystyle\int_{\rho\pi}^{\frac{2\alpha_k\pi}{k-1}}c_{n, k}(x)dx &\text{if } k\in N_{\rho}.
    \end{cases}
\end{align*}
As usual, one can show that $|\tau_k|\to 0$ as $k\to\infty$ using the sandwich theorem. Now, consider first term in \eqref{neweq} as $k\to\infty$:
\begin{align*}
    &\sum_{l=1}^{\lambda_k}\left(\int_{\frac{2l\pi}{k+1}}^{\frac{2l\pi}{k-1}}c_{n, k}(x)dx-\int_{\frac{2(l-1)\pi}{k-1}}^{\frac{2l\pi}{k+1}} c_{n, k}(x)dx\right)\\
    &\qquad =\frac{1}{2^{n-1}}\sum_{m=0}^{(n-1)/2}\binom{n}{m}\sum_{l=1}^{\lambda_k}\left(\int_{\frac{2l\pi}{k+1}}^{\frac{2l\pi}{k-1}}g_{n, m, k}(x)dx-\int_{\frac{2(l-1)\pi}{k-1}}^{\frac{2l\pi}{k+1}}g_{n, m, k}(x)dx\right).
\end{align*}
By applying Lemma \ref{lem2.4} and Lemma \ref{newlem2.6}, we are done. \hfill $\square$

\hfill

\noindent
{\bf Proof of Theorem \ref{thm1.1} ($I_{n, \rho}$, even $n$ case).} We let $\theta_k=\lfloor\rho(k-1)\rfloor$, and define the following two sets:
\begin{align*}
    U_{\rho}&:=\left\{k\in\mathbb{N} : \theta_k\in(\rho(k-1)-1,\rho(k+1)-1]\right\},\\
    V_{\rho}&:=\left\{k\in\mathbb{N} : \theta_k\in(\rho(k+1)-1,\rho(k-1)]\right\}.
\end{align*}
Then $U_{\rho}$ and $V_{\rho}$ partition $\mathbb{N}$.  By solving the equation
		\[\sin^n(kx)-\sin^nx=0\]
		on $[0,\rho\pi]$, we have
		$\sin(kx)-\sin x=0$ or $\sin(kx)+\sin x=0$.
		That is,
		\[2\sin\left(\frac{(k+1)x}{2}\right)\cos\left(\frac{(k-1)x}{2}\right)=0\quad\text{or}\quad2\sin\left(\frac{(k-1)x}{2}\right)\cos\left(\frac{(k+1)x}{2}\right)=0.\]
        This implies
		\[x=\frac{a\pi}{k+1}\quad\text{and}\quad x=\frac{b\pi}{k-1},\]
        where $a, b\in\mathbb{Z}$,
            $1\le a\le \lfloor\rho(k+1)\rfloor$, and  $0\le b\le \lfloor \rho(k-1)\rfloor=\theta_k$.
        \begin{itemize}
		\item If $k\in U_\rho$, then
		\[\frac{\theta_k\pi}{k-1}<\frac{(\theta_k+1)\pi}{k+1}\leq\rho\pi<\frac{(\theta_k+1)\pi}{k-1}\quad\text{and}\quad\frac{(l-1)\pi}{k-1}<\frac{l\pi}{k+1}<\frac{l\pi}{k-1}\]
		for every $l\in\mathbb{Z}$ with $1\leq l\leq\theta_k$, and
		\[s_{n,k}(x)\begin{cases}
			>0&\text{if }x\in\displaystyle\bigsqcup_{l=1}^{\theta_k}\left(\frac{(l-1)\pi}{k-1},\frac{l\pi}{k+1}\right)\sqcup\left(\frac{\theta_k\pi}{k-1},\frac{(\theta_k+1)\pi}{k+1}\right),\\\\
			<0&\text{if }x\in\displaystyle\bigsqcup_{l=1}^{\theta_k}\left(\frac{l\pi}{k+1},\frac{l\pi}{k-1}\right)\sqcup\left(\frac{(\theta_k+1)\pi}{k+1},\rho\pi\right).
		\end{cases}\]
		That is,
		\begin{align*}
			I_{n,k,\rho}
			&=\sum_{l=1}^{\theta_k+1}\left(\int_{\frac{(l-1)\pi}{k-1}}^{\frac{l\pi}{k+1}}s_{n,k}(x)dx-\int_{\frac{l\pi}{k+1}}^{\frac{l\pi}{k-1}}s_{n,k}(x)dx\right)+\int_{\rho\pi}^{\frac{(\theta_k+1)\pi}{k-1}}s_{n,k}(x)dx.
		\end{align*}
          \item 		If $k\in V_\rho$, then
		\[\frac{\theta_k\pi}{k-1}\leq\rho\pi<\frac{(\theta_k+1)\pi}{k+1}\quad\text{and}\quad\frac{(l-1)\pi}{k-1}<\frac{l\pi}{k+1}<\frac{l\pi}{k-1},\]
		for every $l\in\mathbb{Z}$ with $1\leq l\leq\theta_k$, and
		\[s_{n,k}(x)\begin{cases}
			>0&\text{if }x\in\displaystyle\bigsqcup_{l=1}^{\theta_k}\left(\frac{(l-1)\pi}{k-1},\frac{l\pi}{k+1}\right)\sqcup\left(\frac{\theta_k\pi}{k-1},\rho\pi\right),\\\\
			<0&\text{if }x\in\displaystyle\bigsqcup_{l=1}^{\theta_k}\left(\frac{l\pi}{k+1},\frac{l\pi}{k-1}\right).
		\end{cases}\]
		Hence we obtain
		\[I_{n,k,\rho}=\sum_{l=1}^{\theta_k}\left(\int_{\frac{(l-1)\pi}{k-1}}^{\frac{l\pi}{k+1}}s_{n,k}(x)dx-\int_{\frac{l\pi}{k+1}}^{\frac{l\pi}{k-1}}s_{n,k}(x)dx\right)+\int_{\frac{\theta_k\pi}{k-1}}^{\rho\pi}s_{n,k}(x)dx.\]
        \end{itemize}
    By combining the two cases, we have
		\[I_{n,k,\rho}=\sum_{l=1}^{\gamma_k}\left(\int_{\frac{(l-1)\pi}{k-1}}^{\frac{l\pi}{k+1}}s_{n,k}(x)dx-\int_{\frac{l\pi}{k+1}}^{\frac{l\pi}{k-1}}s_{n,k}(x)dx\right)+\varphi_k,\]
		where
		\[\gamma_k:=\begin{cases}
			\theta_k+1&\text{if }k\in U_\rho,\\
            \theta_k&\text{if }k\in V_\rho,
		\end{cases}\quad\text{and}\quad\varphi_k:=\begin{cases}
		\displaystyle\int_{\rho\pi}^{\frac{(\theta_k+1)\pi}{k-1}}s_{n,k}(x)dx&\text{if }k\in U_\rho,\\\\
        \displaystyle\int_{\frac{\theta_k\pi}{k-1}}^{\rho\pi}s_{n,k}(x)dx&\text{if }k\in V_\rho.
		\end{cases}\]
        Then we proceed the same procedure as in the proof of the odd case. That is, we apply Lemma \ref{lem2.3},  Lemma \ref{lem2.4}, Lemma \ref{lem2.5} (part 2), and the sandwich theorem. Finally, we will yield
        \begin{align*}
			&\sum_{l=1}^{\gamma_k}\left(\int_{\frac{(l-1)\pi}{k-1}}^{\frac{l\pi}{k+1}}s_{n,k}(x)dx-\int_{\frac{l\pi}{k+1}}^{\frac{l\pi}{k-1}}s_{n,k}(x)dx\right)\\
			&\qquad=\frac{(-1)^{\frac{n}{2}}}{2^{n-1}}\sum_{m=0}^{(n-2)/2}(-1)^m{n\choose m}\sum_{l=1}^{\gamma_k}\left(\int_{\frac{(l-1)\pi}{k-1}}^{\frac{l\pi}{k+1}}g_{n,m,k}(x)dx-\int_{\frac{l\pi}{k+1}}^{\frac{l\pi}{k-1}}g_{n,m,k}(x)dx\right)\\
			&\qquad\to\displaystyle\frac{(-1)^{\frac{n-2}{2}}}{2^{n-4}\pi}\sum_{m=0}^{(n-2)/2}(-1)^m\binom{n}{m}\frac{(1-\cos((n-2m)\rho\pi))}{(n-2m)^2}\\
            \displaystyle &\qquad\qquad +(-1)^{n/2}\frac{4\rho-1}{2^{n-1}}\sum_{m=0}^{(n-2)/2}(-1)^m\binom{n}{m}\frac{\sin((n-2m)\rho\pi)}{n-2m}
		\end{align*}
        and $|\varphi_k|\to 0$ as $k\to\infty$. Hence the proof is complete. \hfill $\square$

\hfill

\noindent
{\bf Proof of Theorem \ref{thm1.1} ($J_{n, \rho}$, even $n$ case).}
Define $\theta_k=\lfloor\rho(k-1)\rfloor$. We need to solve $\cos^nx=\cos^n(kx)$, equivalently, $\cos x=\pm\cos(kx)$. By the sum-to-product formula, we have
\begin{align*}
    \sin\left(\frac{(k+1)x}{2}\right)\sin\left(\frac{(k-1)x}{2}\right)=0\quad\text{or}\quad \cos\left(\frac{(k+1)x}{2}\right)\cos\left(\frac{(k-1)x}{2}\right)=0.
\end{align*}
So the solutions are $x=a\pi/(k+1)$ and $x=b\pi/(k-1)$, where $a, b\in\mathbb{Z}$ and
\[0\le a\le \lfloor \rho(k+1)\rfloor\quad\text{and}\quad 0\le b\le \lfloor\rho(k-1)\rfloor=\theta_k.\]
Then we may appropriately (details omitted) partition $\mathbb{N}=U_{\rho}\sqcup V_{\rho}$ as in the previous proof.
\begin{itemize}
    \item If $k\in U_{\rho}$, then
    \begin{align*}
        c_{n, k}(x)\begin{cases}
            >0 &\text{if } \displaystyle x\in\bigsqcup_{l=1}^{\theta_k}\left(\frac{l\pi }{k+1}, \frac{l\pi }{k-1}\right)\sqcup\left(\frac{(\theta_k+1)\pi}{k+1},\rho\pi\right),
            \\
            <0 &\text{if } \displaystyle x\in\bigsqcup_{l=1}^{\theta_k+1}\left(\frac{(l-1)\pi}{k-1}, \frac{l\pi }{k+1}\right).
        \end{cases}
    \end{align*}
    \item If $k\in V_{\rho}$, then
    \begin{align*}
        c_{n, k}(x)\begin{cases}
            >0 &\text{if } \displaystyle x\in\bigsqcup_{l=1}^{\theta_k}\left(\frac{l\pi }{k+1}, \frac{l\pi }{k-1}\right),
            \\
            <0 &\text{if } \displaystyle x\in\bigsqcup_{l=1}^{\theta_k}\left(\frac{(l-1)\pi}{k-1}, \frac{l\pi}{k+1}\right)\sqcup\left(\frac{\theta_k\pi}{k-1}, \rho\pi\right).
            \end{cases}
    \end{align*}
\end{itemize}
Merging both cases, we have
\begin{align}\label{neww}
    J_{n, k, \rho}=\sum_{l=1}^{\xi_k}\left(\int_{\frac{l\pi}{k+1}}^{\frac{l\pi}{k-1}}c_{n, k}(x) dx-\int_{\frac{(l-1)\pi}{k-1}}^{\frac{l\pi}{k+1}}c_{n, k}(x) dx\right)+\mu_k,
\end{align}
where
\begin{align*}
    \xi_k:=\begin{cases}
    \theta_k+1 &\text{if } k\in U_{\rho},
    \\
    \theta_k &\text{if } k\in V_{\rho},
    \end{cases}
    \quad\text{and}\quad
    \mu_k:=\begin{cases}
    \displaystyle\int_{\frac{(\theta_k+1)\pi}{k-1}}^{\rho\pi}c_{n, k}(x)dx &\text{if } k\in U_{\rho},\\\\
    \displaystyle\int_{\rho\pi}^{\frac{\theta_k\pi}{k-1}}c_{n, k}(x)dx &\text{if } k\in V_{\rho}.
    \end{cases}
\end{align*}
As usual, when $k\to\infty$, $|\mu_k|\to 0$. Now, the first term of \eqref{neww} can be written as
\begin{align*}
    &\sum_{l=1}^{\xi_k}\left(\int_{\frac{l\pi}{k+1}}^{\frac{l\pi}{k-1}}c_{n, k}(x) dx-\int_{\frac{(l-1)\pi}{k-1}}^{\frac{l\pi}{k+1}}c_{n, k}(x) dx\right)\\
    &\qquad=\frac{1}{2^{n-1}}\sum_{m=0}^{(n-2)/2}\binom{n}{m}\sum_{l=1}^{\xi_k}\left(\int_{\frac{l\pi}{k+1}}^{\frac{l\pi}{k-1}} g_{n, m, k}(x) dx-\int_{\frac{(l-1)\pi}{k-1}}^{\frac{l\pi}{k+1}} g_{n, m, k}(x) dx\right).
\end{align*}
By applying Lemma \ref{lem2.5} (part 2), the proof is complete. \hfill $\square$

\section{Proof of results part II}\label{sec4}

\noindent
{\bf Proof of Theorem \ref{thm1.2} ($I_{n, \rho}$, odd $n$ case).} For $\theta\in\mathbb{R}$, define 
\begin{align*}
    P_n(\theta)&:=\frac{(-1)^{\frac{n-1}{2}}}{2^{n-3}\pi}\sum_{m=0}^{(n-1)/2}(-1)^m{n\choose m}\frac{\sin((n-2m)\theta)}{(n-2m)^2},\\
	Q_n(\theta)&:=(-1)^{\frac{n+1}{2}}\frac{\rho}{2^{n-2}}\sum_{m=0}^{(n-1)/2}(-1)^m{n\choose m}\frac{\cos((n-2m)\theta)}{n-2m},
\end{align*}
which are obviously smooth. Note that $P_n'(\pi/2)=P_n(0)=Q_n(\pi/2)=0$. Consider $P_n$. Differentiating it with respect to $\theta$ twice and using Lemma \ref{lem2.3}, we have
\[P_n^\prime(\theta)=\frac{(-1)^{\frac{n-1}{2}}}{2^{n-3}\pi}\sum_{m=0}^{(n-1)/2}(-1)^m{n\choose m}\frac{\cos((n-2m)\theta)}{n-2m}\]
and
\[P_n^{\prime\prime}(\theta)=-\frac{(-1)^{\frac{n-1}{2}}}{2^{n-3}\pi}\sum_{m=0}^{(n-1)/2}(-1)^m{n\choose m}\sin((n-2m)\theta)=-\frac{4}{\pi}\sin^n\theta.\]
Applying integration by parts, we see that
\begin{align*}
				P_n^\prime(\theta)&=P_n^\prime(\theta)-P_n^\prime\left(\frac{\pi}{2}\right)=\frac{4}{\pi}\int_\theta^{\frac{\pi}{2}}\sin^nydy\\
				&=\frac{4}{\pi}\left(\frac{1}{n}\sin^{n-1}\theta\cos\theta+\left(\frac{n-1}{n}\right)\int_\theta^{\frac{\pi}{2}}\sin^{n-2} ydy\right)\\
				&=\frac{4}{\pi n}\sin^{n-1}\theta\cos\theta+\left(\frac{n-1}{n}\right)P_{n-2}^\prime(\theta).
			\end{align*}
Hence integrating the equality above over $[0, \theta]$ gives
\begin{align*}
				P_n(\theta)=P_n(\theta)-P_n(0)&=\frac{4}{\pi n}\int_0^\theta\sin^{n-1}y\cos ydy+\left(\frac{n-1}{n}\right)(P_{n-2}(\theta)-P_{n-2}(0))\\
				&=\frac{4}{\pi n^2}\sin^n\theta+\left(\frac{n-1}{n}\right)P_{n-2}(\theta).
			\end{align*}
Next, consider $Q_n$. By differentiating $Q_n$ with respect to $\theta$ and utilising Lemma \ref{lem2.3}, we have
\[Q^\prime_n(\theta)=(-1)^{\frac{n-1}{2}}\frac{\rho}{2^{n-2}}\sum_{m=0}^{(n-1)/2}(-1)^m{n\choose m}\sin((n-2m)\theta)=2\rho\sin^n\theta.\]
Applying integration by parts gives
\begin{align*}
				Q_n(\theta)&=Q_n(\theta)-Q_n\left(\frac{\pi}{2}\right)=2\rho\int_{\frac{\pi}{2}}^\theta\sin^nydy\\
				&=-\frac{2\rho}{n}\sin^{n-1}\theta\cos\theta+\left(\frac{2\rho(n-1)}{n}\right)\int_{\frac{\pi}{2}}^\theta\sin^{n-2}ydy\\
				&=-\frac{2\rho}{n}\sin^{n-1}\theta\cos\theta+\left(\frac{n-1}{n}\right)Q_{n-2}(\theta).
			\end{align*}
By adding up $I_{n,\rho}=P_n(\rho\pi)+Q_n(\rho\pi)$, we are done. \hfill $\square$

\hfill

\noindent
{\bf Proof of Theorem \ref{thm1.2} ($J_{n, \rho}$, odd $n$ case).} For $\theta\in\mathbb{R}$, define
\begin{align*}
    P_n(\theta) &:= \frac{1}{2^{n-3}\pi}\sum_{m=0}^{(n-1)/2}\binom{n}{m}\frac{(1-\cos((n-2m)\theta))}{(n-2m)^2},\\
    Q_n(\theta)&:=-\frac{2\rho-1}{2^{n-1}}\sum_{m=0}^{(n-1)/2}\binom{n}{m}\frac{\sin((n-2m)\theta)}{n-2m}.
\end{align*}
Note that $Q_n'(\pi/2)=Q_n(0)=P_n'(0)=P_n(0)=0$.
Then, by performing some differentiation and integration by parts (details omitted), we yield two recursive formulas below. 
\begin{align*}
    P_n(\theta)&=\frac{4}{\pi n^2}(1-\cos^n\theta)+\left(\frac{n-1}{n}\right)P_{n-2}(\theta),\\
    Q_n(\theta)&=-\frac{(2\rho-1)}{n}\sin\theta\cos^{n-1}\theta+\left(\frac{n-1}{n}\right)Q_{n-2}(\theta).
\end{align*}
By adding up $P_n$ and $Q_n$ evaluated at $\theta=\rho\pi$, we are done. \hfill $\square$

\hfill

\noindent
{\bf Proof of Theorem \ref{thm1.2} ($I_{n, \rho}$, even $n$ case).} Similarly, for $\theta\in\mathbb{R}$, define
\begin{align*}
    H_n(\theta)&:=\frac{(-1)^{\frac{n-2}{2}}}{2^{n-4}\pi}\sum_{m=0}^{(n-2)/2}(-1)^m{n\choose m}\frac{1-\cos((n-2m)\theta)}{(n-2m)^2},\\
		K_n(\theta)&:=(-1)^{\frac{n}{2}}\left(\frac{4\rho-1}{2^{n-1}}\right)\sum_{m=0}^{(n-2)/2}(-1)^m{n\choose m}\frac{\sin((n-2m)\theta)}{n-2m}.
\end{align*}
We observe that $H_n(0)=H_n'(0)=K_n(0)=0$. Then, by performing the same arguments, we obtain two recursive formulas below.
\begin{align*}
    H_n(\theta)&=\frac{8}{\pi n^2}\sin^n\theta+\left(\frac{n-1}{n}\right)H_{n-2}(\theta),\\
    K_n(\theta)&=-\left(\frac{4\rho-1}{n}\right)\sin^{n-1}\theta\cos\theta+\left(\frac{n-1}{n}\right)K_{n-2}(\theta).
\end{align*}
Note here that we use the shortcut
\[\frac{\theta}{2^n}\binom{n}{n/2}=\frac{\theta}{2^n}\cdot\frac{n(n-1)(n-2)!}{(n/2)^2((n/2-1)!)^2}=\frac{\theta}{2^{n-2}}\left(\frac{n-1}{n}\right)\binom{n-2}{(n-2)/2}\]
along the way.
By substituting $\theta=\rho\pi$, and adding up $H_n$, $K_n$, we are done. \hfill $\square$

\hfill

\noindent
{\bf Proof of Theorem \ref{thm1.2} ($J_{n, \rho}$, even $n$ case).} For $\theta\in\mathbb{R}$, define
\begin{align*}
    H_n(\theta)&:=\frac{1}{2^{n-4}\pi}\sum_{m=0}^{(n-2)/2}\binom{n}{m}\frac{1-\cos((n-2m)\theta)}{(n-2m)^2},\\
    K_n(\theta)&:=\frac{1-4\rho}{2^{n-1}}\sum_{m=0}^{(n-2)/2}\binom{n}{m}\frac{\sin((n-2m)\theta)}{n-2m}.
\end{align*}
Observe that $H_n(0)=H_n'(0)=K_n(0)=0$. Using the same techniques, we have
\begin{align*}
    H_n(\theta)&=\frac{8}{\pi n^2}(1-\cos^{n}\theta)+\left(\frac{n-1}{n}\right)H_{n-2}(\theta),\\
    K_n(\theta)&=-\left(\frac{4\rho-1}{n}\right)\sin\theta\cos^{n-1}\theta+\left(\frac{n-1}{n}\right)K_{n-2}(\theta).
\end{align*}
Then by summing up $H_n$ and $K_n$ evaluated at $\theta=\rho\pi$, the proof is complete.\hfill $\square$

\section{Proof of results part III}\label{sec5}

\noindent
{\bf Proof of Theorem \ref{thm1.3} (first equation).} First, if $\rho=1/2$, then we are done. So we consider when $\rho\in(1/2, 1]$. Let $k\in\mathbb{N}$, we make the change of variables $u=\pi-x$. Then by the substitution rule, we have
\begin{align*}
    R_{n, k, \rho} =-\int_{\frac{\pi}{2}}^{(1-\rho)\pi}|s_{n, k}(\pi-u)|du=\int_{(1-\rho)\pi}^{\frac{\pi}{2}}|s_{n, k}(\pi-u)|du.
\end{align*}
If $n$ is even, then using the angle addition formula gives 
\begin{align*}
    |s_{n, k}(\pi-u)|=\lvert\sin^n(\pi-u)-\sin^n(k\pi-ku)\rvert=|s_{n, k}(u)|.
\end{align*}
Hence we yield
\begin{align}\label{E1}
    R_{n,k,\rho}=\int_{(1-\rho)\pi}^{\frac{\pi}{2}}\lvert s_{n,k}(u)\rvert du=I_{n,k,1/2}-I_{n,k,1-\rho}.
\end{align}
Taking the limit $k\to\infty$, we obtain the desired result for this case.

If $n$ and $k$ are odd, then $|s_{n, k}(\pi-u)|=|s_{n, k}(u)|$ as well. In this case, we also acquire \eqref{E1}, and taking $k\to\infty$ gives the result we want. 

Now, we consider when $n$ is odd and $k$ is even. Set $\alpha_k:=\lfloor(\rho(k+1)+1)/2\rfloor$, and
			\[f_{n,m,k}(x):=\sin((n-2m)kx)-\sin((n-2m)x).\]
For $k\geq\lfloor(r+5)/(2\rho-1)\rfloor-1$, where $r\in\{0,2\}$, note that $\alpha_k\ge\lceil (k+3)/4\rceil+1$.
We define two sets below:
			\[A_\rho:=\left\{k\in\mathbb{N} : \alpha_k\in\left(\frac{\rho(k+1)-1}{2},\frac{\rho(k-1)}{2}+1\right]\right\},\]
			\[B_\rho:=\left\{k\in\mathbb{N} : \alpha_k\in\left(\frac{\rho(k-1)}{2}+1,\frac{\rho(k+1)+1}{2}\right]\right\}.\]
Note that $A_{\rho}$ and $B_{\rho}$ partition $\mathbb{N}$. Observe that all the roots of $s_{n,k}(x)=0$ on $[\pi/2,\rho\pi]$ are
			\[x=\frac{(2a-1)\pi}{k+1}\quad\text{and}\quad x=\frac{2b\pi}{k-1},\]
			where $a,b\in\mathbb{Z}$, and
            \begin{align*}
                \left\lceil\frac{k+3}{4}\right\rceil\leq a\leq\left\lfloor\frac{\rho(k+1)+1}{2}\right\rfloor=\alpha_k, \quad \left\lceil\frac{k-1}{4}\right\rceil\leq b\leq\left\lfloor\frac{\rho(k-1)}{2}\right\rfloor.
            \end{align*}
            \begin{itemize}
                \item If $k\in A_\rho$, then
			\begin{align*}
				s_{n,k}(x)\begin{cases}
					>0&\text{if }\displaystyle x\in\bigsqcup_{l=\left\lceil\frac{k+3}{4}\right\rceil}^{\alpha_k}\left(\frac{(2l-1)\pi}{k+1},\frac{2(l-1)\pi}{k-1}\right),\\
					&\\
					<0&\text{if }\displaystyle x\in\left(\frac{\pi}{2},\left(2\left\lceil\frac{k+3}{4}\right\rceil-1\right)\frac{\pi}{k+1}\right)\sqcup\bigsqcup_{l=\left\lceil\frac{k+3}{4}\right\rceil}^{\alpha_k-1}\left(\frac{2(l-1)\pi}{k-1},\frac{(2l+1)\pi}{k+1}\right)\\
					&\qquad\qquad\displaystyle\sqcup\left(\frac{2(\alpha_k-1)\pi}{k-1},\rho\pi\right).
				\end{cases}
			\end{align*}
    \item If $k\in B_\rho$, then
			\begin{align*}
				s_{n,k}(x)\begin{cases}
					>0&\text{if }\displaystyle x\in\bigsqcup_{l=\left\lceil\frac{k+3}{4}\right\rceil}^{\alpha_k-1}\left(\frac{(2l-1)\pi}{k+1},\frac{2(l-1)\pi}{k-1}\right)\sqcup\left(\frac{(2\alpha_k-1)\pi}{k+1},\rho\pi\right),\\\\
					<0&\text{if }\displaystyle x\in\left(\frac{\pi}{2},\left(2\left\lceil\frac{k+3}{4}\right\rceil-1\right)\frac{\pi}{k+1}\right)\sqcup\bigsqcup_{l=\left\lceil\frac{k+3}{4}\right\rceil}^{\alpha_k-1}\left(\frac{2(l-1)\pi}{k-1},\frac{(2l+1)\pi}{k+1}\right).
				\end{cases}
			\end{align*}
    \end{itemize}
    By combining both cases, we have
            \begin{align}
			\begin{split}\label{e7}
				R_{n,k,\rho}&=\sum_{l=1}^{\beta_k}\left(\int_{\left(2l+2\left\lceil\frac{k+3}{4}\right\rceil-3\right)\frac{\pi}{k+1}}^{\left(2l+2\left\lceil\frac{k+3}{4}\right\rceil-4\right)\frac{\pi}{k-1}}s_{n,k}(x)dx-\int_{\left(2l+2\left\lceil\frac{k+3}{4}\right\rceil-4\right)\frac{\pi}{k-1}}^{\left(2l+2\left\lceil\frac{k+3}{4}\right\rceil-1\right)\frac{\pi}{k+1}}s_{n,k}(x)dx\right)\\
				&\qquad+\psi_k-\int_{\frac{\pi}{2}}^{\left(2\left\lceil\frac{k+3}{4}\right\rceil-1\right)\frac{\pi}{k+1}}s_{n,k}(x)dx,
			\end{split}
            \end{align}
			where
			\[\beta_k:=\begin{cases}
				\displaystyle\alpha_k-\left\lceil\frac{k+3}{4}\right\rceil+1&\text{if }k\in A_\rho,\\\\
				\displaystyle\alpha_k-\left\lceil\frac{k+3}{4}\right\rceil&\text{if }k\in B_\rho,
			\end{cases}\quad\text{and}\quad\psi_k:=\begin{cases}
			\displaystyle\int_{\rho\pi}^{\frac{(2\alpha_k+1)\pi}{k+1}}s_{n,k}(x)dx&\text{if }k\in A_\rho,\\
			&\\
			\displaystyle\int_{\frac{(2\alpha_k-1)\pi}{k+1}}^{\rho\pi}s_{n,k}(x)dx&\text{if }k\in B_\rho.
			\end{cases}\]
Then by applying Corollary \ref{cor2.2}, Lemma \ref{lem2.3}, Lemma \ref{lem2.4}, and Lemma \ref{lem2.6} (first equation) to the first term of \eqref{e7}, we yield
\begin{align*}
				&\sum_{l=1}^{\beta_k}\left(\int_{\left(2l+2\left\lceil\frac{k+3}{4}\right\rceil-3\right)\frac{\pi}{k+1}}^{\left(2l+2\left\lceil\frac{k+3}{4}\right\rceil-4\right)\frac{\pi}{k-1}}s_{n,k}(x)dx-\int_{\left(2l+2\left\lceil\frac{k+3}{4}\right\rceil-4\right)\frac{\pi}{k-1}}^{\left(2l+2\left\lceil\frac{k+3}{4}\right\rceil-1\right)\frac{\pi}{k+1}}s_{n,k}(x)dx\right)\\
				&\qquad=\frac{(-1)^{\frac{n-1}{2}}}{2^{n-1}}\sum_{m=0}^{(n-1)/2}(-1)^m{n\choose m}\times\\
				&\qquad\qquad\sum_{l=1}^{\beta_k}\left(\int_{\left(2l+2\left\lceil\frac{k+3}{4}\right\rceil-3\right)\frac{\pi}{k+1}}^{\left(2l+2\left\lceil\frac{k+3}{4}\right\rceil-4\right)\frac{\pi}{k-1}}f_{n,m,k}(x)dx-\int_{\left(2l+2\left\lceil\frac{k+3}{4}\right\rceil-4\right)\frac{\pi}{k-1}}^{\left(2l+2\left\lceil\frac{k+3}{4}\right\rceil-1\right)\frac{\pi}{k+1}}f_{n,m,k}(x)dx\right)\\
				&\qquad\to\frac{(-1)^{\frac{n-1}{2}}}{2^{n-1}}\sum_{m=0}^{(n-1)/2}(-1)^m{n\choose m}\left(\frac{4}{\pi(n-2m)^2}\left((-1)^{\frac{n-2m-1}{2}}-\sin((n-2m)\rho\pi)\right)\right.\\
				&\qquad\qquad-\left.\frac{2(1-\rho)}{n-2m}\cos((n-2m)\rho\pi)\right),
			\end{align*}
			as $k\to\infty$ and $k\in A_\rho,B_\rho$, since $n-2m$ is odd for $m\in\mathbb{N}$. Then one can use the sandwich theorem to show that the second and the third terms on \eqref{e7} tend to $0$ as $k\to\infty$. Thus
            \begin{align*}
				R_{n,\rho}
				&=\frac{1}{2^{n-3}\pi}\sum_{m=0}^{(n-1)/2}{n\choose m}\frac{1}{(n-2m)^2}\\
				&\qquad-\left(\frac{(-1)^{\frac{n-1}{2}}}{2^{n-3}\pi}\sum_{m=0}^{(n-1)/2}(-1)^m{n\choose m}\frac{\sin((n-2m)\rho\pi)}{(n-2m)^2}\right.\\
				&\qquad\qquad+\left.(-1)^{\frac{n-1}{2}}\left(\frac{1-\rho}{2^{n-2}}\right)\sum_{m=0}^{(n-1)/2}(-1)^m\binom{n}{m}\frac{\cos((n-2m)\rho\pi)}{n-2m}\right).
			\end{align*}
            Since
			\[I_{n,1/2}=\frac{1}{2^{n-3}\pi}\sum_{m=0}^{(n-1)/2}{n\choose m}\frac{1}{(n-2m)^2},\]
			 $n$ is odd, and $1-\rho\in[0,1/2]$, we have
			\begin{align*}
				I_{n,1-\rho}
				&=\frac{(-1)^{\frac{n-1}{2}}}{2^{n-3}\pi}\sum_{m=0}^{(n-1)/2}(-1)^m{n\choose m}\frac{\sin((n-2m)\rho\pi)}{(n-2m)^2}\\
				&\qquad+(-1)^{\frac{n-1}{2}}\left(\frac{1-\rho}{2^{n-2}}\right)\sum_{m=0}^{(n-1)/2}(-1)^m{n\choose m}\frac{\cos((n-2m)\rho\pi)}{n-2m}.
			\end{align*}
            Thus we are done. \hfill $\square$

\hfill

\noindent
{\bf Proof of Theorem \ref{thm1.3} (second equation).} Similar to the proof of the first equation, the case $n$ is even, and the case $n, k$ are odd can be justified easily. Thus let us consider when $n$ is odd, and $k$ is even. Set $\beta_k:=\lfloor\rho(k+1)/2\rfloor$, and
\[g_{n,m,k}(x):=\cos((n-2m)kx)-\cos((n-2m)x).\]
For $k\geq\lfloor(7-r)/(2\rho-1)\rfloor-1$, where $r\in\{0,2\}$, we have $\beta_k\ge\lceil (k+1)/4\rceil+1$ (easy to verify).
We define two sets below:
			\[A_\rho:=\left\{k\in\mathbb{N} : \beta_k\in\left(\frac{\rho(k+1)}{2}-1,\frac{\rho(k-1)}{2}\right]\right\},\]
			\[B_\rho:=\left\{k\in\mathbb{N} : \beta_k\in\left(\frac{\rho(k-1)}{2},\frac{\rho(k+1)}{2}\right]\right\}.\]
Then $A_{\rho}$ and $B_{\rho}$ partition $\mathbb{N}$. Observe that all  roots of $c_{n,k}(x)=0$ on $[\pi/2,\rho\pi]$ are
			\[x=\frac{2a\pi}{k+1}\quad\text{and}\quad x=\frac{2b\pi}{k-1},\]
			where $a,b\in\mathbb{Z}$, and
            \begin{align*}
                \left\lceil\frac{k+1}{4}\right\rceil\leq a\leq\left\lfloor\frac{\rho(k+1)}{2}\right\rfloor=\beta_k, \quad \left\lceil\frac{k-1}{4}\right\rceil\leq b\leq\left\lfloor\frac{\rho(k-1)}{2}\right\rfloor.
            \end{align*}
            
            \textit{Scenario 1: $k\equiv0\pmod 4$.}
            \begin{itemize}
                \item If $k\in A_\rho$, then
			\begin{align*}
				c_{n,k}(x)\begin{cases}
					>0&\text{if }\displaystyle x\in\left(\frac{\pi}{2},\frac{k\pi}{2(k-1)}\right)\sqcup\bigsqcup_{l=\frac{k}{4}+1}^{\beta_k}\left(\frac{2l\pi}{k+1},\frac{2l\pi}{k-1}\right),\\
					&\\
					<0&\text{if }\displaystyle x\in\bigsqcup_{l=\frac{k}{4}}^{\beta_k-1}\left(\frac{2l\pi}{k-1},\frac{2(l+1)\pi}{k+1}\right)\sqcup\left(\frac{2\beta_k\pi}{k-1},\rho\pi\right).
				\end{cases}
			\end{align*}
    \item If $k\in B_\rho$, then 
			\begin{align*}
				c_{n,k}(x)\begin{cases}
					>0&\text{if }\displaystyle x\in\left(\frac{\pi}{2},\frac{k\pi}{2(k-1)}\right)\sqcup\bigsqcup_{l=\frac{k}{4}+1}^{\beta_k-1}\left(\frac{2l\pi}{k+1},\frac{2l\pi}{k-1}\right)\sqcup\left(\frac{2\beta_k\pi}{k+1},\rho\pi\right),\\\\
					<0&\text{if }\displaystyle x\in\bigsqcup_{l=\frac{k}{4}}^{\beta_k-1}\left(\frac{2l\pi}{k-1},\frac{2(l+1)\pi}{k+1}\right).
				\end{cases}
			\end{align*}
    \end{itemize}
    By combining both cases, we have
            \begin{align}
			\begin{split}\label{equa3}
				T_{n,k,\rho}&=\sum_{l=0}^{\delta_k}\left(\int_{\left(2l+\frac{k}{2}\right)\frac{\pi}{k+1}}^{\left(2l+\frac{k}{2}\right)\frac{\pi}{k-1}}c_{n,k}(x)dx-\int_{\left(2l+\frac{k}{2}\right)\frac{\pi}{k-1}}^{\left(2l+\frac{k}{2}+2\right)\frac{\pi}{k+1}}c_{n,k}(x)dx\right)\\
				&\qquad+\psi_k-\int_{\frac{k\pi}{2(k+1)}}^{\frac{\pi}{2}}c_{n,k}(x)dx,
			\end{split}
            \end{align}
			where
			\[\delta_k:=\begin{cases}
				\displaystyle\beta_k-\frac{k}{4}&\text{if }k\in A_\rho,\\\\
				\displaystyle\beta_k-\frac{k}{4}-1&\text{if }k\in B_\rho,
			\end{cases}\quad\text{and}\quad\psi_k:=\begin{cases}
			\displaystyle\int_{\rho\pi}^{\frac{2(\beta_k+1)\pi}{k+1}}c_{n,k}(x)dx&\text{if }k\in A_\rho,\\
			&\\
			\displaystyle\int_{\frac{2\beta_k\pi}{k+1}}^{\rho\pi}c_{n,k}(x)dx&\text{if }k\in B_\rho.
			\end{cases}\]
            Note that $\lceil(k+1)/4\rceil=(k/4)+1$. Then by applying Corollary \ref{cor2.2}, Lemma \ref{lem2.3}, Lemma \ref{lem2.4}, and Lemma \ref{lem2.6} (second equation) to the first term of \eqref{equa3}, we yield
\begin{align*}
				&\sum_{l=0}^{\delta_k}\left(\int_{\left(2l+\frac{k}{2}\right)\frac{\pi}{k+1}}^{\left(2l+\frac{k}{2}\right)\frac{\pi}{k-1}}c_{n,k}(x)dx-\int_{\left(2l+\frac{k}{2}\right)\frac{\pi}{k-1}}^{\left(2l+\frac{k}{2}+2\right)\frac{\pi}{k+1}}c_{n,k}(x)dx\right)\\
				&\qquad=\frac{1}{2^{n-1}}\sum_{m=0}^{(n-1)/2}{n\choose m}\times\sum_{l=0}^{\delta_k}\left(\int_{\left(2l+\frac{k}{2}\right)\frac{\pi}{k+1}}^{\left(2l+\frac{k}{2}\right)\frac{\pi}{k-1}}g_{n,m,k}(x)dx-\int_{\left(2l+\frac{k}{2}\right)\frac{\pi}{k-1}}^{\left(2l+\frac{k}{2}+2\right)\frac{\pi}{k+1}}g_{n,m,k}(x)dx\right)\\
				&\qquad\to\frac{1}{2^{n-1}}\sum_{m=0}^{(n-1)/2}{n\choose m}\left(-\frac{4}{\pi(n-2m)^2}\cos((n-2m)\rho\pi)+\frac{(1-2\rho)}{n-2m}\sin((n-2m)\rho\pi)\right)
			\end{align*}
			as $k\to\infty$ and $k\in A_\rho,B_\rho$, since $n-2m$ is odd for $m\in\mathbb{N}$. Then one can use the sandwich theorem to show that the second and the third terms on \eqref{equa3} tend to $0$ as $k\to\infty$. Thus
            \begin{align*}
				T_{n,\rho}
				&=-\frac{1}{2^{n-3}\pi}\sum_{m=0}^{(n-1)/2}{n\choose m}\frac{\cos((n-2m)\rho\pi)}{(n-2m)^2}+\frac{(1-2\rho)}{2^{n-1}}\sum_{m=0}^{(n-1)/2}{n\choose m}\frac{\sin((n-2m)\rho\pi)}{n-2m}.
			\end{align*}
            By knowing the value of $J_{n, 1/2}$,
			 $n$ is odd, and $1-\rho\in[0,1/2]$, we have
			\begin{align*}
				J_{n,1-\rho}
				&=\frac{1}{2^{n-3}\pi}\sum_{m=0}^{(n-1)/2}{n\choose m}\frac{(1+\cos((n-2m)\rho\pi))}{(n-2m)^2}\\
                &\qquad-\frac{(1-2\rho)}{2^{n-1}}\sum_{m=0}^{(n-1)/2}{n\choose m}\frac{\sin((n-2m)\rho\pi)}{n-2m}.
			\end{align*}
            
            \textit{Scenario 2: $k\equiv2\pmod 4$}.
            \begin{itemize}
                \item If $k\in A_\rho$, then
			\begin{align*}
				c_{n,k}(x)\begin{cases}
					>0&\text{if }\displaystyle x\in\bigsqcup_{l=\frac{k+2}{4}}^{\beta_k}\left(\frac{2l\pi}{k+1},\frac{2l\pi}{k-1}\right),\\
					&\\
					<0&\text{if }\displaystyle x\in\left(\frac{\pi}{2},\frac{(k+2)\pi}{2(k+1)}\right)\sqcup\bigsqcup_{l=\frac{k+2}{4}}^{\beta_k-1}\left(\frac{2l\pi}{k-1},\frac{2(l+1)\pi}{k+1}\right)\sqcup\left(\frac{2\beta_k\pi}{k-1},\rho\pi\right).
				\end{cases}
			\end{align*}
    \item If $k\in B_\rho$, then
			\begin{align*}
				c_{n,k}(x)\begin{cases}
					>0&\text{if }\displaystyle x\in\bigsqcup_{l=\frac{k+2}{4}}^{\beta_k-1}\left(\frac{2l\pi}{k+1},\frac{2l\pi}{k-1}\right)\sqcup\left(\frac{2\beta_k\pi}{k+1},\rho\pi\right),\\\\
					<0&\text{if }\displaystyle x\in\left(\frac{\pi}{2},\frac{(k+2)\pi}{2(k+1)}\right)\sqcup\bigsqcup_{l=\frac{k+2}{4}}^{\beta_k-1}\left(\frac{2l\pi}{k-1},\frac{2(l+1)\pi}{k+1}\right).
				\end{cases}
			\end{align*}
    \end{itemize}
    Combining both cases, we have
            \begin{align}
			\begin{split}\label{e3}
				T_{n,k,\rho}&=\sum_{l=1}^{\delta_k}\left(\int_{\left(2l+\frac{k+2}{2}-2\right)\frac{\pi}{k+1}}^{\left(2l+\frac{k+2}{2}-2\right)\frac{\pi}{k-1}}c_{n,k}(x)dx-\int_{\left(2l+\frac{k+2}{2}-2\right)\frac{\pi}{k-1}}^{\left(2l+\frac{k+2}{2}\right)\frac{\pi}{k+1}}c_{n,k}(x)dx\right)\\
				&\qquad+\psi_k-\int_{\frac{\pi}{2}}^{\frac{(k+2)\pi}{2(k+1)}}c_{n,k}(x)dx,
			\end{split}
            \end{align}
			where
			\[\delta_k:=\begin{cases}
				\displaystyle\beta_k-\frac{k+2}{4}+1&\text{if }k\in A_\rho,\\\\
				\displaystyle\beta_k-\frac{k+2}{4}&\text{if }k\in B_\rho,
			\end{cases}\quad\text{and}\quad\psi_k:=\begin{cases}
			\displaystyle\int_{\rho\pi}^{\frac{2(\beta_k+1)\pi}{k+1}}c_{n,k}(x)dx&\text{if }k\in A_\rho,\\
			&\\
			\displaystyle\int_{\frac{2\beta_k\pi}{k+1}}^{\rho\pi}c_{n,k}(x)dx&\text{if }k\in B_\rho.
			\end{cases}\]
Note that $\lceil(k+1)/4\rceil=(k+2)/4$. Then by 
repeating the exact same argument, we will obtain the same result as in Scenario 1. \hfill $\square$

\begin{remark}
    It is quite astonishing that the two equations turn out to look the same, although we currently cannot find a direct implication between these two.
\end{remark}

\noindent
{\bf Proof of Corollary \ref{cor1.4}.} 
We will only prove the formulas for $I_{n, \rho}$ when $n$ is odd, since other parts can be done similarly.
First, by using Theorem \ref{thm1.3},  we have
\begin{align}\label{nice}
    I_{n, \rho}=I_{n, 1/2}+R_{n, \rho}=2I_{n, 1/2}-I_{n, 1-\rho}
\end{align}
for $n\in\mathbb{N}$. Next, by Theorem \ref{thm1.2}, we have
\begin{align*}
I_{n, 1/2}=\left(\frac{n-1}{n}\right)I_{n-2, 1/2}+\frac{4}{\pi n^2}.
\end{align*}
We also see that $1-\rho\in[0, 1/2]$, and so
\begin{align*}
    I_{n,1-\rho}&=\left(\frac{n-1}{n}\right)I_{n-2,1-\rho}-\frac{2(1-\rho)}{n}\sin^{n-1}((1-\rho)\pi)\cos((1-\rho)\pi)+\displaystyle\frac{4}{\pi n^2}\sin^n((1-\rho)\pi).
\end{align*}
Finally, viewing \eqref{nice} with $n\mapsto n-2$, and performing some algebraic manipulations, the proof is complete. \hfill $\square$

\subsection{Special case $I_{n, 1}=J_{n, 1}$}

We end this section with an interesting phenomenon.


\begin{corollary}
		For $n\in\mathbb{N}$, we have $I_{n, 1}=J_{n, 1}$. 
	\end{corollary}
\noindent
{\bf Proof.} Note that the value of $J_{n, 1}$ is $\mathcal{A}_n$ in the introduction. From Corollary \ref{cor1.4}, $I_{n, 1}=2I_{n, 1/2}$ (the cases $n=1,2$ are also trivially true). Next, by Theorem \ref{thm1.1}, we have
\[I_{n,1}=\begin{cases}
				\displaystyle\frac{1}{2^{n-4}\pi}\sum_{m=0}^{(n-1)/2}{n\choose m}\frac{1}{(n-2m)^2}&\text{if $n$ is odd,}\\
				&\\
				\displaystyle\frac{(-1)^{\frac{n-2}{2}}}{2^{n-5}\pi}\sum_{m=0}^{(n-2)/2}(-1)^m{n\choose m}\frac{(1-(-1)^{\frac{n-2m}{2}})}{(n-2m)^2}&\text{if $n$ is even.}
			\end{cases}\]
If $n\equiv 2\pmod{4}$, then $(n-2m)/2\equiv 1-m\pmod{2}$. So
\begin{align*}
    (-1)^{\frac{n-2m}{2}}=\begin{cases}
        1 &\text{if $m$ is odd,}\\
        -1 &\text{if $m$ is even.}
    \end{cases}
\end{align*}
This gives
\[I_{n,1}=\frac{1}{2^{n-6}\pi}\sum_{\substack{m=0\\m\text{ is even}}}^{(n-2)/2}{n\choose m}\frac{1}{(n-2m)^2}=\frac{1}{2^{n-4}\pi}\sum_{m=0}^{(n-2)/4}{n\choose2m}\frac{1}{(n/2-2m)^2}.\]
On the other hand, if $n\equiv 0\pmod{4}$, then $(n-2m)/2\equiv -m\pmod{2}$. So $(-1)^{\frac{n-2m}{2}}$ equals $1$ when $m$ is even, and equals $-1$ when $m$ is odd.
Thus we deduce
\begin{align*}
				I_{n,1}&=-\frac{1}{2^{n-6}\pi}\sum_{\substack{m=1\\m\text{ is odd}}}^{(n-2)/2}(-1){n\choose m}\frac{1}{(n-2m)^2}
                =\frac{1}{2^{n-6}\pi}\sum_{m=0}^{(n-4)/4}{n\choose2m+1}\frac{1}{(n-2(2m+1))^2}.
			\end{align*}
Hence we are done. \hfill $\square$


\section{Proof of results part IV}\label{sec6}

\noindent
{\bf Proof of Theorem \ref{thm1.5}.} Verifying the formulas for $I_{n, \rho}$ and $J_{n, \rho}$ is essentially the same.
First, we claim that for any $m\in\mathbb{N}$ and $y\in\mathbb{R}$,
\[\int_{m\pi}^{y\pi}\lvert s_{n,k}(x)\rvert dx=\begin{cases}
			\displaystyle\int_0^{(y-m)\pi}\lvert s_{n,k}(x)\rvert dx&\text{if $m$ is even},\\
			&\\
			\displaystyle\int_{(m+1-y)\pi}^\pi\lvert s_{n,k}(x)\rvert dx&\text{if $m$ is odd}.
		\end{cases}\]
Indeed, if $m$ is even, then we make the change of variables $t=x-m\pi$. This gives
\[\int_{m\pi}^{y\pi}\lvert s_{n,k}(x)\rvert dx=\int_0^{(y-m)\pi}\lvert s_{n,k}(m\pi+t)\rvert dt=\int_0^{(y-m)\pi}\lvert s_{n,k}(t)\rvert dt.\]
However, if $m$ is odd we make the change of variables $t=(m+1)\pi-x$. This implies
\[\int_{m\pi}^{y\pi}\lvert s_{n,k}(x)\rvert dx=-\int_{\pi}^{(m+1-y)\pi}\lvert s_{n,k}(t)\rvert dt=\int_{(m+1-y)\pi}^{\pi}\lvert s_{n,k}(t)\rvert dt.\]
Now let us prove the theorem. If $\rho\in[0, 1)$, then $\rho=\{\rho\}$ and $\lfloor\rho\rfloor=0$ and the result follows from previous developments. Now, if $\rho\ge 1$, then for each $k\in\mathbb{N}$, we have
\begin{align}\label{e1}
    I_{n, k, \rho}=\sum_{m=0}^{\lfloor \rho\rfloor-1}\int_{m\pi}^{(m+1)\pi}|s_{n, k}(x)|dx+\int_{\lfloor \rho\rfloor\pi}^{\rho\pi}|s_{n, k}(x)|dx.
\end{align}
For the first term of \eqref{e1}, using our earlier claim, we have
\[\int_{m\pi}^{(m+1)\pi}\lvert s_{n,k}(x)\rvert dx=\int_0^\pi\lvert s_{n,k}(x)\rvert dx=I_{n,k,1}\]
		for every $m\in\{0,\ldots,\lfloor\rho\rfloor\}$. Hence
        \[\sum_{m=0}^{\lfloor \rho\rfloor-1}\int_{m\pi}^{(m+1)\pi}\lvert s_{n,k}(x)\rvert dx=\lfloor\rho\rfloor I_{n,k,1}.\]
For the second term of \eqref{e1}, using our claim again, we obtain
\begin{align*}
			\int_{\lfloor\rho\rfloor\pi}^{\rho\pi}\lvert s_{n,k}(x)\rvert dx&=\begin{cases}
			\displaystyle\int_0^{(\rho-\lfloor\rho\rfloor)\pi=\{\rho\}\pi}\lvert s_{n,k}(x)\rvert dx&\text{if $\lfloor\rho\rfloor$ is even},\\
			&\\
			\displaystyle\int_{(\lfloor\rho\rfloor+1-\rho)\pi=(1-\{\rho\})\pi}^\pi\lvert s_{n,k}(x)\rvert dx&\text{if $\lfloor\rho\rfloor$ is odd}.		\end{cases}\\
			&=\begin{cases}
				I_{n,k,\{\rho\}}&\text{if $\lfloor\rho\rfloor$ is even},\\\\
				I_{n,k,1}-I_{n,k,1-\{\rho\}}&\text{if $\lfloor\rho\rfloor$ is odd}.
			\end{cases}
		\end{align*}
        This implies
		\[I_{n,k,\rho}=\begin{cases}
			\lfloor\rho\rfloor I_{n,k,1}+I_{n,k,\{\rho\}}&\text{if $\lfloor\rho\rfloor$ is even},\\\\
			(\lfloor\rho\rfloor+1)I_{n,k,1}-I_{n,k,1-\{\rho\}}&\text{if $\lfloor\rho\rfloor$ is odd}.
			\end{cases}\]
		for every $k,n\in\mathbb{N}$. Finally, by taking the limit $k\to\infty$, we are done. \hfill $\square$

\section{From rational to irrational $\rho$}\label{sec7}

So far, we have established our formulas only for rational values of $\rho$. To extend these results to irrational $\rho$, we will use the following standard result in analysis.

\begin{lemma}[Moore-Osgood theorem, \cite{EZ}]
Let $\{a_{j, k}\}_{(j, k)\in\mathbb{N}^2}$ be a double sequence of real numbers. Suppose that
$\lim_{j\to\infty} a_{j,k} = b_k$
uniformly in $k$, and that
$\lim_{k\to\infty} a_{j,k} = c_j$
for all sufficiently large $j$. Then both $\lim_{k\to\infty} b_k$ and $\lim_{j\to\infty} c_j$ exist and are equal to the iterated limits, i.e.,
\[
\lim_{j, k\to\infty}a_{j, k}=
\lim_{k\to\infty}\lim_{j\to\infty} a_{j,k}
=
\lim_{j\to\infty}\lim_{k\to\infty} a_{j, k}.
\]
\end{lemma}


We now justify why the formulas remain valid for irrational $\rho$. Fix an irrational $\rho\in(0, 1/2)$, and choose a sequence $\{\rho_j\}_{j\ge1}\subset\mathbb{Q}$ such that $\rho_j\to\rho$ as $j\to\infty$. For each $j,k\in\mathbb{N}$, define
\[
a_{j,k} := \int_0^{\rho_j\pi} \lvert s_{n,k}(x)\rvert dx.
\]
For each fixed $k$, we estimate
\[
\left|
\int_0^{\rho\pi}\lvert s_{n,k}(x)\rvert dx - a_{j,k}
\right|
=
\left|
\int_{\rho_j\pi}^{\rho\pi}\lvert s_{n,k}(x)\rvert dx
\right|
\le
2\pi\,\lvert \rho-\rho_j\rvert,
\]
since $\lvert s_{n,k}(x)\rvert\le 2$ for all $x$. Hence
\[
\lim_{j\to\infty} a_{j,k}
=
\int_0^{\rho\pi}\lvert s_{n,k}(x)\rvert dx
= I_{n,k,\rho}
\]
and the convergence in $j$ is uniform in $k$.
On the other hand, for each fixed $j$, our previous work for rational $\rho$ shows that
\[
\lim_{k\to\infty} a_{j,k}
=
\lim_{k\to\infty}\int_0^{\rho_j\pi}\lvert s_{n,k}(x)\rvert dx
=
I_{n,\rho_j},
\]
where $I_{n,\rho_j}$ is given explicitly by Theorem \ref{thm1.1}.
Therefore, by the Moore-Osgood theorem, 
$
\lim_{j\to\infty}\lim_{k\to\infty} a_{j,k}
=
\lim_{k\to\infty}\lim_{j\to\infty} a_{j,k}.
$
In particular,
\[
\lim_{j\to\infty}I_{n, \rho_j}=\lim_{j\to\infty}\lim_{k\to\infty}
\int_0^{\rho_j\pi}\lvert s_{n,k}(x)\rvert dx
=
\lim_{k\to\infty}\int_0^{\rho\pi}\lvert s_{n,k}(x)\rvert dx
=
I_{n,\rho},
\]
and the left-hand side is exactly the explicit expression for $I_{n,\rho}$ obtained in Theorem \ref{thm1.1}, now evaluated at the limit $\rho$ (by continuity). The same argument applies when $\lvert s_{n,k}\rvert$ is replaced by $\lvert c_{n,k}\rvert$, so all the formulas for $J_{n,\rho}$ extend from rational to irrational $\rho$ as well.

\section*{Acknowledgements}

First, we would like to thank the referees for their constructive comments, which significantly enhance the shape of this paper. 
We also give a million thanks to Prof. Ratinan Boonklurb for pointing out some gaps in the original draft. Last but not least, we would like to thank our families for the continuous support. The  first author is funded by Kamnoetvidya Science Academy, while the second author is funded by the Philip K. H. Wong Foundations Scholarship at the University of Hong Kong.

\section*{Conflict of interest}

The authors declare no conflict of interest.

\end{document}